\documentclass[final,onefignum,onetabnum]{siamart171218}
\usepackage{amsmath,amsfonts,amssymb}
\usepackage{graphicx,color}

\usepackage{a4wide}
\usepackage{enumerate,float}

\usepackage[utf8]{inputenc}
\usepackage{pgf, tikz}
\usetikzlibrary{arrows}
\usetikzlibrary{patterns}
\usetikzlibrary{decorations.pathmorphing,datavisualization}
\renewcommand{\L}[1]{\mathbf{L^#1}}
\newcommand{\Lloc}[1]{\mathbf{L^{#1}_{loc}}}
\newcommand{\Cc}[1]{\mathbf{C^{#1}_{c}}}
\newcommand{\C}[1]{\mathbf{C^{#1}}}

\newcommand{\Wloc}[2]{\mathbf{W^{#1,#2}_{loc}}}
\newcommand{\bv}{\mathbf{BV}}
\newcommand{\bvloc}{\mathbf{BV_{loc}}}

\newsiamremark{remark}{Remark}
\newsiamremark{example}{Example}
\newcommand{\D}{\mathcal{D}}
\newcommand{\R}{\mathcal{R}}

\DeclareMathOperator{\tv}{Tot.Var.}
\DeclareMathOperator{\m}{meas}
\def\bega{\begin{array}}
\def\enda{\end{array}}
\def\begi{\begin{itemize}}
\def\endi{\end{itemize}}
\def\bel{\begin{equation}\label}
\def\eeq{\end{equation}}
\DeclareMathOperator\sign{sign}

\renewcommand{\theequation}{\thesection.\arabic{equation}}
\renewcommand{\thesection}{\arabic{section}}

\allowdisplaybreaks

\title{\bf Vanishing Viscosity and Backward Euler
  Approximations for Conservation
  Laws with Discontinuous Flux}

\author{Graziano Guerra\thanks{Department of Mathematics and its Applications,
  University of Milano - Bicocca (\email{graziano.guerra@unimib.it})}
\and Wen Shen\thanks{Department of Mathematics, Penn State University,
University Park, PA 16802, U.S.A.
 (\email{wxs27@psu.edu})}}

\begin{document}

\maketitle

\begin{abstract}
  Solutions to a class of one dimensional conservation laws with
  discontinuous flux
  are constructed relying on the
  Crandall-Liggett theory of nonlinear contractive
  semigroups~\cite{BP,CL}, with a vanishing viscosity approach.  
  The solutions to the corresponding viscous conservation laws are studied
  using the Backward Euler approximations. 
  We prove their convergence to a \emph{unique}
  vanishing viscosity solution to the Cauchy problem for the non
  viscous equations as the viscous parameter tends to zero.  
  This approach allows to avoid the technicalities in existing literature
  such as traces, Riemann problems, interfaces conditions, compensated
  compactness and entropy inequalities.
  Consequently we establish our result under very mild assumptions on the flux,
  with only a requirement on the smoothness  with
  respect to the unknown variable and
  a condition that allows the application of the maximum principle.
\end{abstract}



\section{Introduction }
\label{sec:1}
\setcounter{equation}{0}

We consider the Cauchy problem for the scalar conservation law 
\begin{equation}\label{eq:sc}
  u_t + f(x, u)_x =0,
\end{equation}
with initial data 
\begin{equation}\label{eq:sci}
u(0,x)=\bar u(x). 
\end{equation}
In the simpler case where $f=f(u)$ is independent of $x$, solutions have been
constructed by a variety of techniques \cite{Dafermosfirst,Dafermos,Kru}.  In
particular, in \cite{C72} it was proved that the abstract theory of
nonlinear contractive semigroups developed by Crandall and Liggett
\cite{CL} can indeed be applied to scalar conservation laws, and
yields the same solutions obtained by Kruzhkov \cite{Kru} as vanishing
viscosity limits. While these approaches are effective even for multi--dimensional
scalar conservation laws, their exploitation is harder when the flux
depends explicitly on the time and space variables $(t,x)$ in a discontinuous way.
Aim of the present paper is to develop a  semigroup approach for the 
one--dimensional case in the more general context
where the flux function $f=f(x,u)$  is allowed to depend on $x$ in a
discontinuous way, by extending the classical results of~\cite{C72,CL}.

\bigskip

We consider the following hypotheses on the flux $f$:
  \begin{itemize}
\item[\textbf{f0)}] 
  \begin{enumerate}[i)]
  \item
    $x\mapsto f(x,\omega)$ is in
    $\L{\infty}\left(\mathbb{R},\mathbb{R}\right)$
    for any $\omega\in\mathbb{R}$;
    $\omega\mapsto f(x,\omega)$ is smooth for any
    $x\in\mathbb{R}$;
  \item
    there exists a constant $L\ge 0$ such that, for any fixed
    $x\in\mathbb{R}$:
    \begin{displaymath}
        \left|f\left(x,\omega_{1}\right)-f\left(x,\omega_{2}\right)\right|\le
        L\left|\omega_{1}-\omega_{2}\right|,\quad\text{ for any }\omega_{1},\omega_{2}\in\mathbb{R};
    \end{displaymath}
  \item
    there exists a constant $L_{1}\ge 0$ such that, 
    \begin{displaymath}
        \int_{\mathbb{R}}\left|f\left(x,0\right)\right|\; dx\le
        L_{1}.
    \end{displaymath}
  \end{enumerate}
\item[\textbf{f1)}]
  The flux $f$ satisfies \textbf{f0)}, 
  and has the following form
\begin{displaymath}
  f\left(x,\omega\right)=
  \begin{cases}
    f_{l}\left(\omega\right) &\text{ if } x\le 0,\\
    f_{r}\left(\omega\right) &\text{ if } x>0    ,
  \end{cases}
\end{displaymath}
where $f_{l}$ and $f_{r}$ are smooth functions satisfying
\begin{displaymath}
  f_{l}(0)=f_{r}(0)=0,\qquad f_{l}(1)=f_{r}(1).
\end{displaymath}
\end{itemize}

\bigskip

Scalar conservation laws with discontinuous flux arise in many applications 
where the conservation laws describe physical models in rough media. 
Examples include but are not limited to
traffic flow with rough road condition and various polymer flooding
models in two phase flow in porous media. 
Beginning with the work by Isaacson \& Temple~\cite{IT92,IT95,IT86,T82}
and by Risebro and collaborators~\cite{GimseRisebro,GimseRisebro2,KR}, 
scalar conservation laws with discontinuous coefficients have become the topic
of a vast literature~\cite{AG1,AKR1,AndreianovKarslenRisebro,BV1,
  CR,D1,GNPT1,GimseRisebro,JohWin,KR2,M1,SV1}.

The existence of solutions for~\eqref{eq:sc} can be established through a compactness
argument on a family of approximate solutions.
These approximations can be constructed by
mollification of the flux~\cite{BV1,KR,Ostrov},  
by wave front tracking~\cite{GNPT1,GimseRisebro,GimseRisebro2,Gue1,KR2}, 
by Godunov's method~\cite{AG2,AmaGue1,IT95,KarTow,LonTemJin}, 
and by several other numerical schemes~\cite{BurKarTow1,KarTow2,SV1,Tow00,Tow01}.
We would also like to mention  the
recent related results on existence of solutions for Cauchy problems 
for models of polymer flooding~\cite{MR3487271}
and slow erosion in granular flow~\cite{MR3384834}. 

In a general setting,  the solutions to the conservation law~\eqref{eq:sc}
can be obtained as limits of two combined approximations:
\begin{equation}\label{eq:approx2}
u_t + f^{\delta_n} (x,u)_x = {\varepsilon_n} u_{xx}.
\end{equation}
Here $\varepsilon_n u_{xx}$ is a viscosity term, and $ f^{\delta_n} (x,u)$ 
is a mollified flux which is smooth in $x$. 
As $n\to \infty$, one takes the double limits 
$\varepsilon_n \to 0$ and 
$\delta_n \to 0$ (where  $f^{\delta_n} \to f^0=f$).
It is important to observe that in  general these two limits do not commute. 
Indeed, one can let $\varepsilon_n ,\delta_n \to 0$ keeping the ratio
$\kappa = \delta_n/\varepsilon_n$ constant. 
A detailed study of viscous traveling waves in \cite{GS2016,MR3663611} 
reveals that, for the same initial data,  infinitely many  limit solutions of (\ref{eq:approx2})
can exist, depending  on the ratio $\kappa$. 
The uniqueness of the double-limit solution is proved in \cite{MR3663611}  only under 
some additional  monotonicity conditions on the  flux function and on the 
mollification $f^{\delta_n}$.

In this paper we set $\delta_n \equiv 0$, and consider the viscous approximation to~\eqref{eq:sc}:
\begin{equation}\label{eq:vsc}
u_t + f(x, u)_x = \varepsilon u_{xx},
\end{equation}
for small $\varepsilon>0$. 
A Backward Euler scheme is adopted to generate approximate solutions
to the viscous equation~\eqref{eq:vsc}.
Using the results in~\cite{BP,CL} and relying on a detailed study of
the Backward Euler approximations, 
we establish existence and uniqueness of the vanishing viscosity limit,
as $\varepsilon \to 0$. 

We remark that the backward Euler approximation 
was recently implemented in~\cite{BressanShen},
to construct a semigroup of solutions to  a conservation law with nonlocal flux,
modeling slow erosion phenomena in granular flow.

In the  literature, uniqueness of solutions to
\eqref{eq:sc} is usually obtained through specific entropy conditions, possibly
supplemented with interface conditions at the point 
where the flux is discontinuous,
satisfied by the limit of the approximate solutions, 
see~\cite{MisVee,AndreianovKarslenRisebro,BurKarTow1,BurKarKenTow,GNPT1,
  KarTow2,KarTow,KR,Kru}.
It must be noted that, when the flux is discontinuous in the variable $x$, 
different entropy conditions or interface conditions may lead to different
solutions to~\eqref{eq:sc}.
This is also indicated by the non-uniqueness of the double limits
for~\eqref{eq:approx2}, studied in~\cite{MR3663611}. 
A systematic study of the various entropy conditions that can be imposed
on the solutions to~\eqref{eq:sc}, leading to different
semigroups of solutions, can be found in~\cite{AndreianovKarslenRisebro}.

In addition to the vanishing viscosity approach, an additional approach 
to obtain uniqueness is available in the literature, 
utilizing the so called \emph{adapted entropies}.
The basic concept was first introduced in~\cite{AP1}, and
then further extended and applied in~\cite{BaiJen,Gwiazda1,CheEveKli,Panov}.
Under further restrictions on the flux function, the adapted entropy inequality can
be applied to multi-dimensional problems~\cite{Gwiazda1,Gwiazda2}. 
However, with the exception of  some very particular fluxes, 
the solutions selected by the adapted entropies in~\cite{AP1}
are NOT
the vanishing viscosity solutions obtained
by letting  $\varepsilon\to 0$ in~\eqref{eq:vsc}. 
Some preliminary analysis shows 
that the adapted entropy concept corresponds to 
taking $\varepsilon_n \to 0$ first, then taking $\delta_n \to 0$ in~\eqref{eq:approx2}. 
A detailed discussion can be found in Section~\ref{sec:Counter},
where counter examples and analysis for selected examples are provided,
and more observations are made.

The novelty of our approach lies mainly on the techniques applied to the
problem, i.e.~the application of the
Brezis \& Pazy convergence result~\cite{BP} to  obtain the existence and
uniqueness of vanishing viscosity solutions to conservation
laws with discontinuous fluxes. 
Compactness and entropy arguments are only used to study  solutions to the
resolvent equations, constructing the approximate and the limit
semigroups (see Section~\ref{sec:vv}). This involves solutions to
ordinary differential equations, depending only on the variable $x$. 
With this approach and using the result in~\cite{BP}, we
prove directly the strong convergence of the
semigroups generated by~\eqref{eq:vsc} to a unique semigroup generated
by~\eqref{eq:sc}, without the need of additional entropy conditions.
In this way, we
obtain the uniqueness and strong convergence results 
without any additional hypothesis on the flux. 
We list a few comparisons with some existing literature.

\vspace{0.1cm}
\begin{itemize}
\item 
We do not require  the nondegeneracy condition, 
which is usually required for compensated compactness
arguments~\cite{KarRasTad,KarTow2}.
\item  
We do not have requirement on the shape of the graph of the
flux, which is usually needed by the arguments based on BV
bounds~\cite{MisVee,BurKarTow1,GNPT1,IT92,KarTow}.
In particular, we do not exclude the presence of an
infinite number of maxima/minima or flux crossings, 
which was required by~\cite{KarTow}. 
\item
We study the convergence, as $\varepsilon\to 0$, of
solutions to~\eqref{eq:vsc} directly, 
without mollifying the flux as it is done, for
instance, in~\cite{CheEveKli,KarRasTad,KarRisTow02}. 
Therefore we avoid the
problem of choosing the relative ratio of convergence between
the mollification parameter and the
viscosity. 
\end{itemize}

\vspace{0.2cm}
In this paper we establish the existence and uniqueness of solution for the 
conservation law where the flux is discontinuous at one location. 
Such a result can serve as a building block for equations where the 
discontinuities in the flux function form a more complex pattern. 
Indeed, the result in this paper 
is utilized as the starting point for the recent paper~\cite{BGS}, 
where the existence and uniqueness of the vanishing viscosity limit 
is extended to 
one dimensional scalar conservation laws with regulated flux.
%
%
To be precise, in~\cite{BGS} we prove the 
existence and uniqueness of the limit as $\varepsilon\to 0$
of the solution $u^{\varepsilon}$ to
\begin{displaymath}
  \begin{cases}
    u_t + f(t,x,u)_x~=~\varepsilon u_{xx}\,,\\
    u(0,x)~=~u_0 (x).
  \end{cases}
\end{displaymath}
Here the mapping $(t,x)\mapsto f$ 
is a regulated function in two dimensions (see Definition~1.1 in~\cite{BGS}), 
which can be highly discontinuous in the $(t,x)$--plane. 
Specially, this result can be applied directly to the existence and 
uniqueness of solutions to the triangular system
\[
  \begin{cases}
    u_t + f(v,u)_x~=~0\,,\\
    v_t + g(v)_x ~=~ 0,
  \end{cases}
  \qquad 
    \begin{cases}
    u(0,x)=u_0(x),\\
    v(0,x)=v_0(x),
  \end{cases}
\]
as the vanishing viscosity solution of 
\[
  \begin{cases}
    u_t + f(v,u)_x~=~\varepsilon u_{xx}\,,\\
    v_t + g(v)_x ~=~ 0,
  \end{cases}
  \qquad 
    \begin{cases}
    u(0,x)=u_0(x),\\
    v(0,x)=v_0(x),
  \end{cases}
\]
under  mild  assumptions on the flux $g$ and the initial data $v_0(x)$. 
We refer to~\cite{BGS} for details.

\bigskip

The remainder of the paper is organized as follows. 
In Section~\ref{sec:rev} we review classical results on non linear
semigroups that are used in the other sections. In
Section~\ref{sec:vis} we study the \emph{resolvent equation}
\begin{displaymath}
  u+\lambda\left[f(x,u)_{x}-\varepsilon u_{xx}\right]=w
\end{displaymath}
for the viscous problem and prove that, under the assumption
\textbf{f0)},
it has a unique solution
$u=J_{\lambda}^{\varepsilon}w$.
Furthermore, according to~\cite{CL}, the operator $J_{\lambda}^{\varepsilon}$ 
generates a non linear semigroup $S^{\varepsilon}_{t}$  of weak 
solutions for the viscous equation~\eqref{eq:vsc}. 
In Section~\ref{sec:vv}, under the hypothesis \textbf{f1)},
we show that $J_{\lambda}^{\varepsilon}w$, as $\varepsilon\to 0$,
converges to a unique limit $J_{\lambda}w$
which solves
\begin{displaymath}
  u+\lambda f(x,u)_{x}=w.
\end{displaymath}
In Section~\ref{sec:vvs} we apply the results in~\cite{CL} to show
that 
$J_{\lambda}$ generates a
non linear semigroup
 $S_{t}$ whose
trajectories are solutions to~\eqref{eq:sc}.
Then~\cite{BP} is applied to show that
$S^{\varepsilon}_{t}$ converges to $S_{t}$ uniformly for $t$ in
compact sets. 
See the diagram in Figure~\ref{fig:diagram}.
In Section 6 we discuss in some detail the adapted entropies introduced in~\cite{AP1},
to illustrate their difference from vanishing viscosity solutions. 
Finally, several examples and counterexamples related to the
generation of non linear semigroups are presented in Section~7,
together with some final remarks.


\begin{figure}[htbp]
\begin{center}
\setlength{\unitlength}{1mm}
\begin{picture}(60,37)(0,2)  
\put(5,5){\fbox{$S^\varepsilon_t$}}
\put(53,5){\fbox{$S_t$}}
\put(5,30){\fbox{$J^\varepsilon_\lambda$}}
\put(53,30){\fbox{$J_\lambda$}}
\put(12,6){\vector(1,0){40}}\put(25,7){\small Section 5}\put(25,2){ ($\varepsilon \to 0$)} 
\put(7.5,27){\vector(0,-1){17}} \put(8,18){($\lambda\to 0$)}
\put(-10,19){\small Section 3} \put(-10,15){\small with f0)}
\put(12,32){\vector(1,0){40}}\put(20,33){\small Section 4, with f1)}\put(25,28){ ($\varepsilon \to 0$)} 
\put(55,27){\vector(0,-1){17}} \put(56,18){($\lambda\to 0$)}
\put(40,19){\small Section 5} 
\end{picture}
\caption{A diagram for the structure of the paper.}
\label{fig:diagram}
\end{center}
\end{figure}
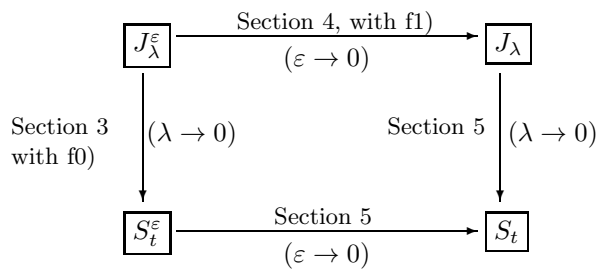

\section{Review on contractive semigroups generated by backward Euler operator}
\setcounter{equation}{0}
\label{sec:rev}
We first give a brief review on the main results in \cite{BP,CL}, 
which are important to the analysis in this paper. 
Let $X$ be a Banach space with norm $\|\cdot\|$, and let 
$A$ be a possibly nonlinear, multivalued map that we view as a subset
of $X\times X$.
The set $Au$, the domain of $A$ and its range are defined as
\begin{equation}
  \label{eq:domDef}
  \begin{split}
    Au &=~\left\{v\in X: \left(u,v\right)\in A\right\},\\
    \D (A) &=~\{ u\in X\,:~~Au\not= \emptyset~\},\\
    \R (A) &=~\bigcup_{u\in\D(A)} Au.
  \end{split}
\end{equation}
We say that the operator $A$ is \textbf{accretive} if 
\begin{equation}\label{accr}
  v_1\in A u_1,~~v_2\in A u_2, ~~\lambda>0 
  \quad\implies\quad 
  \bigl\| (u_1+\lambda v_1)-  (u_2+\lambda v_2)\bigr\|~\geq~\|u_1 - u_2\|\,.
\end{equation}

Consider the abstract Cauchy problem
\begin{equation}\label{1}
  \frac{d}{ dt} u + A u ~\ni~0\,,
  \qquad\qquad 
  u(0)= \bar u\,.
\end{equation} 
We define its \textbf{Backward Euler  operator} $J_\lambda$ by setting
\begin{equation}\label{Jl}
  \left(w,u\right)\in J_\lambda
  \quad\hbox{if and only if$\quad u\in \D (A)$ ~and    there exists
    $v\in Au$ such that $u + \lambda v = w$}.
\end{equation}
If $A$ is accretive, because of~\eqref{accr}, $J_{\lambda}$ is a single valued map.
Fix a time step $\lambda>0$, we consider the approximation
\begin{equation}\label{be2}
u(t+\lambda) ~=~J_\lambda u(t)\,.  
\end{equation}
Approximate solutions to \eqref{1} can be constructed by time iterations
with the Backward Euler operator. For time interval $[0,\tau]$
and  $n$ time steps, one computes
\begin{displaymath}
  u(\tau) ~\approx~(J_\lambda)^n  \bar u\,, \qquad \lambda=\tau/n. 
\end{displaymath}

The Backward Euler operator $J_\lambda$ has the following properties. 
\begin{lemma} \cite[Lemma~1.2]{CL}
 Let $A$ be an accretive operator on the Banach space $X$,
and assume that there exists $\lambda_0>0$ such that
\begin{equation}\label{RD} 
  \D\left(J_{\lambda}\right)=\R (I+\lambda A)~\supseteq~\overline{\D (A)}
  \qquad\qquad\forall \lambda\in \,]0,\lambda_0]\,.
\end{equation}
Then, for $\lambda, \mu\in \, ]0,\lambda_0]$ the  following holds.
\begin{itemize}
\item[(i)] 
The operator $J_\lambda$ is single-valued. Indeed, 
for $u_1,u_2\in \D (J_\lambda)$
\begin{equation}\label{6}
\left\|J_\lambda u_1 - J_\lambda u_2\right\|~\leq~\left\|u_1-u_2\right\|\,.
\end{equation}
\item[(ii)] 
For $u\in \D (A)$ one has
\begin{equation}\label{7} 
\frac{1}{\lambda}\,\|J_\lambda u - u\|~\leq~\|Au\|~\doteq~\inf\{\|v\|\,,~~v\in Au\}.
\end{equation}
\item[(iii)] 
If $n$ is a positive integer and $u\in \D (J_{\lambda})$, then 
\begin{equation}\label{8} 
\left\|(J_\lambda)^n u - u\right\|~\leq~n\, \left\|J_\lambda u - u\right\|\,.
\end{equation}
\item[(iv)] 
For any $u\in \D (J_\lambda)$, the 
``resolvent formula'' holds:
\begin{equation}\label{9} 
\tilde u\;\doteq\;
\frac{\mu}{\lambda} u + \frac{\lambda-\mu}{\lambda} J_\lambda u\in \D (J_\mu)
\qquad\mbox{and}\qquad
J_\lambda u =J_\mu\left( \tilde u\right).
\end{equation}
\end{itemize}
\end{lemma}

Here and in the following $I$ denotes the identity operator.
The Backward Euler approximation converges to a limit as $n\to\infty$.
The limit solution generates a semigroup of contractions, as shown in this 
elegant result by Crandall \& Liggett \cite{CL}. 

\begin{theorem} \label{th:CL} 
  \cite[Theorem~I]{CL}
 Let $A$ be an accretive operator on the Banach space $X$, and 
 $J_\lambda$  the corresponding Backward-Euler operator. 
Assume that there exists $\lambda_0>0$ such that
\begin{equation}\label{range}
  \D\left(J_{\lambda}\right)=
  \R (I+\lambda A)~\supseteq~\overline{\D (A)}\qquad\qquad\forall \lambda\in \,]0,\lambda_0]\,.\end{equation}
Then the following holds.
\begin{itemize}
\item[{\bf (I)}] For every initial datum
  $ u\in \overline{\D (A)}$ and every $t\geq 0$ the limit
  \begin{equation}\label{lim}
    S_t  u~\doteq~
    \lim_{n\to \infty} \left(J_{t/n} \right)^{n}  u
  \end{equation}
  is well defined.

\item[{\bf (II)}] The family of operators $\{S_t\,;~t\geq 0\}$ defined
  at (\ref{lim}) is a continuous semigroup of contractions on the set
  $\overline{\D (A)}$.  Namely
  \begin{itemize}
  \item[(i)] For all $t,s\geq 0 $ and $u\in \overline{\D (A)}$
    one has~~$S_0 u = u$,~ $S_tS_s u = S_{t+s} u$.
  \item[(ii)] For every $u\in \overline{\D (A)}$ the map
    $t\mapsto S_t u$ is continuous.
  \item[(iii)] For every $u_1,u_2\in \overline{\D (A)}$ and every
    $t\geq 0$ one has $\|S_t u_1 - S_t u_2\|\leq \|u_1-u_2\|$.
  \end{itemize}
\end{itemize}
\end{theorem}

Consider a family of  accretive operators $A^\sigma$, 
and   the corresponding semigroups $S^{\sigma}$. 
As shown by Brezis and Pazy \cite{BP}, 
the limit 
$\lim_{\sigma\to 0} S_t^{\sigma} x$
exists if one has the convergence of 
the corresponding Backward Euler operators $J^{\sigma}_\lambda$.

\begin{theorem}
  \label{th:BP}
  \cite[Theorem~3.1]{BP}
  Let $A, A^{\sigma}$, $\sigma\in\,]0, \sigma_0]$  be accretive operators such that
  \begin{displaymath}
    \D\left(J_{\lambda}\right)=\R (I+\lambda A)\supseteq\overline{\D(A)}\,,
    \quad
    \D\left(J_{\lambda}^{\sigma}\right)=\R (I+\lambda
    A^{\sigma})\supseteq\overline{\D (A^{\sigma})}\,,
    \quad\forall \lambda\in]0,\lambda_{0}],\;
    \forall \sigma\in]0,\sigma_{0}].
\end{displaymath}
Let $S, S^{\sigma}$ be the corresponding semigroups (Theorem~\ref{th:CL}), and call 
\begin{displaymath}
  D~\doteq~\bigcap_{\sigma>0}\overline{\D (A^{\sigma} )}\cap \D (A).
\end{displaymath}
If the corresponding Backward Euler operators satisfy
\begin{equation}\label{LBE}
\lim_{\sigma\to 0} J_\lambda^{\sigma} u ~=~J_\lambda u\qquad
\forall  u\in \overline D\,,\;
\forall \lambda\in]0,\lambda_{0}],
\end{equation}
then 
\begin{equation}\label{ls2}
\lim_{\sigma\to 0} S^{\sigma}_t u ~=~S_t u\qquad\qquad\forall u\in 
\overline D,\; \forall t\ge 0,
\end{equation}
and the limit is uniform for $t$ in bounded intervals.
\end{theorem}

\section{The resolvent equation for the viscous problem}
\setcounter{equation}{0}
\label{sec:vis}
In this section we assume hypothesis \textbf{f0)} for the flux $f$, 
and study  the resolvent equation for the viscous
conservation law~\eqref{eq:vsc}. 
We establish suitable properties  so that the classical results 
stated in Section 2 can be applied. 
To this end, we consider  
\begin{displaymath}
  u_{t}+\left[f(x,u)-\varepsilon u_{x}\right]_{x}\ni 0
\end{displaymath}
and define the non linear map $A^{\varepsilon}\subset
  \L{1}\left(\mathbb{R},\mathbb{R}\right)
  \times \L{1}\left(\mathbb{R},\mathbb{R}\right)$ as
  \begin{equation}
    \label{eq:Aepsilon_def}
    \left(u,v\right)\in A^{\varepsilon} \quad\text{ if and only if }
    \quad u,\;v\in\L{1}\left(\mathbb{R},\mathbb{R}\right)
    \text{ and }\left[f(x,u)-\varepsilon u_{x}\right]_{x}=v.
  \end{equation}
The domain of the map is
  \begin{equation}
    \label{eq:Aepsilon_dom}
    \mathcal{D}\left(A^{\varepsilon}\right) =
    \left\{u\in\L{1}\left(\mathbb{R},\mathbb{R}\right):
      \ \left[f(x,u)-\varepsilon u_{x}\right]_{x}\in
      \L{1}\left(\mathbb{R},\mathbb{R}\right)\right\}.
  \end{equation}
We also use the notation $A^{\varepsilon}u$ to denote $v$
  in~\eqref{eq:Aepsilon_def}, since $A^{\varepsilon}$ is a single valued operator.

Recall that $I$ is
the identity, $\lambda$ is any positive real number and $\mathcal{R}\left(B\right)$ denotes the
range of a map $B$. 
We consider the resolvent equation 
\begin{equation}
  \label{eq:res_eq}
 u+\lambda A^{\varepsilon}u = w \quad\text{ i.e. }\quad u + \lambda\left[f(x,u)- \varepsilon u_{x}\right]_{x}=w,
\end{equation}
where $w$ is any given function in
$\L1\left(\mathbb{R},\mathbb{R}\right)$. 

We begin with the definition of weak solution to ~\eqref{eq:res_eq}. 

\begin{definition}\label{def:weaksolution}
A function $u\in\Lloc1(\mathbb{R},\mathbb{R})$ is a \em{weak solution} 
to~\eqref{eq:res_eq}, with data $w\in
\L1\left(\mathbb{R},\mathbb{R}\right)$,
if for any test function $\phi\in \mathbf{C}^2_c
\left(\mathbb{R},\mathbb{R}\right)$ it holds
\[
\int_{\mathbb{R}}\left[ \frac{u(x)-w(x)}{\lambda} \phi(x) - f(x,u) \phi'(x) - \varepsilon u(x)\phi''(x) \right]\;dx =0.
\] 
\end{definition}

We now  introduce the definitions of
the upper and lower solutions to \eqref{eq:res_eq}.

\begin{definition}
  \label{def:lower_upper}
  Let $\Omega\subset \mathbb{R}$ be open and
  $w$ be any function in $\Lloc1\left(\Omega,\mathbb{R}\right)$.
  The function
  $u\in\Lloc1\left(\Omega,\mathbb{R}\right)$ is a \textbf{lower solution} to
  \eqref{eq:res_eq} in $\Omega$ if
  \begin{displaymath}
    u + \lambda\left[f(x,u)-\varepsilon u_{x}\right]_{x}\le w
  \end{displaymath}
  holds in $\Omega$ in the sense of distributions.
 On the other hand, the function
  $u\in\Lloc1\left(\Omega,\mathbb{R}\right)$ is an \textbf{upper solution} to
  \eqref{eq:res_eq} in $\Omega$ if
  \begin{displaymath}
    u + \lambda\left[f(x,u)-\varepsilon u_{x}\right]_{x}\ge w
  \end{displaymath}
  holds instead.
\end{definition}

Lower and upper solutions to \eqref{eq:res_eq} satisfy the following
maximum principle.

\begin{theorem}
  \label{thm:max_princ}
  Let $u_{1},\;u_{2}\in\Lloc1\left(\Omega,\mathbb{R}\right)$ be
  respectively a lower and an upper solution to \eqref{eq:res_eq} in
  the open set $\Omega$ with right hand sides respectively equal to
  $w_{1},\;w_{2}\in\Lloc1\left(\Omega,\mathbb{R}\right)$:
  \begin{equation}
    \label{eq:up_low}
    \begin{cases}
      u_{1}+\lambda\left[f\left(x,u_{1}\right)-\varepsilon u_{1,x}\right]_{x}\le w_{1},\\
      u_{2}+\lambda\left[f\left(x,u_{2}\right)-\varepsilon u_{2,x}\right]_{x}\ge w_{2}.
    \end{cases}
  \end{equation}
  Let $a,b\in\mathbb{R}\cup\left\{\pm\infty\right\}$ be such that
  $\left]a,b\right[\subset \Omega$ and
  \begin{equation}
    \label{eq:liminf_prop}
    \liminf_{x\to a^{+}}\left[u_{1}(x)-u_{2}(x)\right]\le 0,\quad 
    \liminf_{x\to b^{-}}\left[u_{1}(x)-u_{2}(x)\right]\le 0,
  \end{equation}
  then the inequality
  \begin{equation}
    \label{eq:max_estimate}
    \int_{a}^{b}\left[u_{1}(x)-u_{2}(x)\right]^{+}\;dx \le
    \int_{a}^{b}\left[w_{1}(x)-w_{2}(x)\right]^{+}\; dx
  \end{equation}
  holds,
  where $\left[\cdot\right]^{+}$ denotes
  the positive part of a real
  number: $\left[t\right]^{+}=\max\left\{0,t\right\}$.
  In particular, if $w_{1}\le w_{2}$ holds in $\left]a,b\right[$, then $u_{1}\le
  u_{2}$ holds in the same interval.
\end{theorem}

\begin{proof}
  Define the function $v=u_{1}-u_{2}$. Subtracting the
  inequalities in \eqref{eq:up_low} we have 
  \begin{equation}
    \label{eq:v_ineq}
    v + \lambda\left[f\left(x,u_{1}\right)-f\left(x,u_{2}\right) -
    \varepsilon v_{x}\right]_{x}\le w_{1}-w_{2},
\end{equation}
in the space of distributions. 
Therefore, the distribution 
\[w_{1}-w_{2}-v- \lambda\left[f\left(x,u_{1}\right)-f\left(x,u_{2}\right) -
    \varepsilon v_{x}\right]_{x}\] 
    is non negative and consequently
  a positive Radon measure on $\Omega$ (see \cite[Theorem
  2.14]{Rudin}). Since $w_{1}-w_{2}-v$ is a locally integrable function,
  $\left[f\left(x,u_{1}\right)-f\left(x,u_{2}\right)-\varepsilon v_{x}\right]_{x}$ is a
  signed Radon measure. Therefore 
  \begin{displaymath}
f\left(x,u_{1}\right)-f\left(x,u_{2}\right)-\varepsilon v_{x}    \in\bvloc\left(\Omega,\mathbb{R}\right)
    \subset \Lloc1\left(\Omega,\mathbb{R}\right).
  \end{displaymath}
Hypothesis \textbf{f0)} implies that
  $f(x,u_{1}),f(x,u_{2})\in\Lloc1\left(\Omega,\mathbb{R}\right)$,
   leading to 
  $v_{x}\in\Lloc1\left(\Omega,\mathbb{R}\right)$. 
  Thus,   $v$ is locally absolutely continuous in $\Omega$.

If $v\le 0$ holds in $\left]a,b\right[$, there is nothing to prove. 
Otherwise, let
$\left]\alpha,\beta\right[$ be any connected component of the open set
\[V=\left\{x\in\left]a,b\right[ :\; v(x)>0\right\}.\]
Note that we do not exclude
the possibilities $\alpha=-\infty$ and $\beta=+\infty$.
Now, hypothesis~\eqref{eq:liminf_prop} and the continuity of $v$ in
$\Omega$ imply
\begin{equation}
  \label{eq:liminf_v}
  \liminf_{x\to \alpha^{+}}v(x)= 0,\quad 
  \liminf_{x\to \beta^{-}}v(x)= 0,\quad v(x) > 0 \text{ for all }
  x\in\left]\alpha,\beta\right[.
\end{equation}

Let $\mathcal{L}$ be the intersection of the Lebesgue points in
$\Omega$ of the functions $f\left(x,u_{1}\right)$,
$f\left(x,u_{2}\right)$ and $v_{x}$. Fix
$\eta,\;\xi\in\mathcal{L}$ such that $\alpha< \eta < \xi < \beta$,
we evaluate the measures in~\eqref{eq:v_ineq} over the set
$]\eta,\xi[$ and obtain
\begin{equation}
  \begin{split}
    \int_{\eta}^{\xi}v(x)\; dx &+ \lambda\left[f\left(\xi,u_{1}(\xi)\right)
    -f\left(\xi,u_{2}(\xi)\right)\right] \\
    & -\lambda\left[f\left(\eta,u_{1}(\eta)\right)
    -f\left(\eta,u_{2}(\eta)\right)\right]-\lambda\varepsilon
  \left[v_{x}\left(\xi\right)
  -v_{x}\left(\eta\right)\right]\\
&\le \int_{\eta}^{\xi}\left[w_{1}(x)-w_{2}(x)\right]\;dx
\le \int_{\eta}^{\xi}\left[w_{1}(x)-w_{2}(x)\right]^{+}\;dx .
  \end{split}
\end{equation}
Using \textbf{f0)} and recalling that $v>0$ in $\left]\alpha,\beta\right[$, 
the above inequality becomes
\begin{equation}
  \label{eq:before_claim}
    \int_{\eta}^{\xi}v(x)\; dx \le \lambda\left[\varepsilon v_{x}\left(\xi\right)
    +  L v\left(\xi\right)\right] +\lambda \left[-\varepsilon v_{x}\left(\eta\right)
    +L v\left(\eta\right)\right]
+ \int_{\eta}^{\xi}\left[w_{1}(x)-w_{2}(x)\right]^{+}\;dx .
\end{equation}

Now we claim that
\begin{equation}
  \label{eq:claim}
  \liminf_{\overset{\eta\to \alpha^{+}}{\eta\in\mathcal{L}}}\left[-\varepsilon v_{x}\left(\eta\right)
    + L v\left(\eta\right)\right]\le 0,\qquad
  \liminf_{\overset{\xi\to \beta^{-}}{\xi\in\mathcal{L}}}\left[\varepsilon v_{x}\left(\xi\right)
    + L v\left(\xi\right)\right]\le 0. 
\end{equation}
Indeed, by contradiction, suppose that the second inequality
in~\eqref{eq:claim} were not true (the other case being similar).
Then there should exist $\gamma>0$ and $\xi_{o}<\beta$ such that
$\varepsilon v_{x}\left(\xi\right)
  + L v\left(\xi\right)\ge\gamma$ for all
$\xi\in\left(\xi_{o},\beta\right)\cap\mathcal{L}$. Solving this differential
inequality with $v\left(\xi_{o}\right)>0$ as initial data we have
\begin{displaymath}
  v\left(\xi\right)\ge
  \left[v\left(\xi_{o}\right)-\frac{\gamma}{L}\right]
  e^{-\frac{L}{\varepsilon}\left(\xi-\xi_{o}\right)}+\frac{\gamma}{L}.
\end{displaymath}
If in this last inequality we take the lower limit as $\xi\to\beta$
we have (including the case $\beta=+\infty$):
\begin{displaymath}
    \liminf_{\xi\to\beta^{-}}v\left(\xi\right)\ge
    v\left(\xi_{o}\right)e^{-\frac{L}{\varepsilon}\left(\beta-\xi_{o}\right)}+
    \frac{\gamma}{L}
    \left(1 -
      e^{-\frac{L}{\varepsilon}\left(\beta-\xi_{o}\right)}\right)>0
\end{displaymath}
which contradicts \eqref{eq:liminf_v}. 

Now we take the lower limits in \eqref{eq:before_claim} as
$\eta\to\alpha^{+}$ and $\xi\to\beta^{-}$.  Using
\eqref{eq:claim} we obtain
  \begin{displaymath}
    \int_{\alpha}^{\beta}v(x)\;dx \le
    \int_{\alpha}^{\beta}\left[w_{1}(x)-w_{2}(x)\right]^{+}\; dx.
  \end{displaymath}
  
  Finally, writing $V$ as the union of its connected components
  $V=\bigcup_{i=1}^{N}\left(\alpha_{i},\beta_{i}\right)$ (with $N=+\infty$
  if there are countable many connected components), we compute
  \begin{equation}
    \begin{split}
      \int_{a}^{b}&\left[u_{1}(x)-u_{2}(x)\right]^{+}\;dx=
      \int_{a}^{b}\left[v(x)\right]^{+}\;dx
      =\int_{V}v(x)\; dx\\
      &=\sum_{i=1}^{N}\int_{\alpha_{i}}^{\beta_{i}}v(x)\; dx
      \le
      \sum_{i=1}^{N}\int_{\alpha_{i}}^{\beta_{i}}\left[w_{1}(x)-w_{2}(x)\right]^{+}\;
    dx
    \le \int_{a}^{b}\left[w_{1}(x)-w_{2}(x)\right]^{+}\; dx,
    \end{split}
  \end{equation}
  proving the theorem.
\end{proof}

If the functions $u$ and $w$ in \eqref{eq:res_eq} are
integrable over $\mathbb{R}$, then  the operator $A^{\varepsilon}$ defined
in~\eqref{eq:Aepsilon_def} is accretive. 
Indeed we have the following result.

%

\begin{corollary}
  \label{cor:contraction_eq}
  Let
  $u_{1},\;u_{2},\;w_{1},\;w_{2}\in\L1\left(\mathbb{R},\mathbb{R}\right)$
  satisfy
  \begin{equation}
    \label{eq:contraction_eq1}
    \begin{cases}
      u_{1}+\lambda\left[f\left(x,u_{1}\right)-\varepsilon u_{1,x}\right]_{x}= w_{1},\\
      u_{2}+\lambda\left[f\left(x,u_{2}\right)-\varepsilon u_{2,x}\right]_{x}= w_{2},
    \end{cases}
  \end{equation}
  in the sense of distribution, then the following inequalities hold:
  \begin{align}
    \label{eq:max_princ_eq2}
    \int_{\mathbb{R}}\left[u_{1}(x)-u_{2}(x)\right]^{+}\; dx
    &\le \int_{\mathbb{R}}\left[w_{1}(x)-w_{2}(x)\right]^{+}\; dx,\\
    \label{eq:contraction_eq2}
    \int_{\mathbb{R}}\left|u_{1}(x)-u_{2}(x)\right|\; dx
    &\le \int_{\mathbb{R}}\left|w_{1}(x)-w_{2}(x)\right|\; dx.
  \end{align}
\end{corollary}
\begin{proof}
  According to Definition~\ref{def:lower_upper},
  $u_{1}$ and $u_{2}$ are both lower and upper solutions
  to~\eqref{eq:res_eq} with right hand side respectively $w_{1}$ and $w_{2}$.
  Since
  they are integrable, we have
  \[\liminf_{x\to \pm\infty}\left[u_{1}(x)-u_{2}(x)\right]\le 0.\]
  Theorem~\ref{thm:max_princ} with $a=-\infty$ and
  $b=+\infty$ can be applied to get~\eqref{eq:max_princ_eq2}.
  Changing the role of $u_{1}$ and $u_{2}$ allows us to
  obtain~\eqref{eq:contraction_eq2}.
\end{proof}

In the following theorem 
we establish some properties of the
Backward Euler operator
$J_{\lambda}^{\varepsilon}=\left(I+\lambda
  A^{\varepsilon}\right)^{-1}$ defined in~\eqref{Jl} and 
show that $\mathcal{R}\left(I+\lambda
  A^{\varepsilon}\right)=
\L1\left(\mathbb{R},\mathbb{R}\right)$.

  \begin{theorem}
      \label{lem:has_sol0}
      Suppose the flux $f(x,\omega)$ satisfies hypothesis \textbf{f0)}. Then
      for any $\lambda$, $\varepsilon>0$,
      $w\in\L1\left(\mathbb{R},\mathbb{R}\right)$, 
      there exists a
  unique weak solution
  $u=J_{\lambda}^{\varepsilon}w\in\L1\left(\mathbb{R},\mathbb{R}\right)$ 
  to equation
  \eqref{eq:res_eq}. The maps
  $A^{\varepsilon}:\D\left(A^{\varepsilon}\right)
  \to \L1\left(\mathbb{R},\mathbb{R}\right)$ and
  $J_{\lambda}^{\varepsilon}:\L1\left(\mathbb{R},\mathbb{R}\right)
  \to \L1\left(\mathbb{R},\mathbb{R}\right)$
  satisfy
  \begin{enumerate}[(i)]
  \item
    $J_{\lambda}^{\varepsilon}w_{1}\le J_{\lambda}^{\varepsilon}w_{2}$
    whenever $w_{1}\le w_{2}$, (monotonicity);
  \item
    $\displaystyle{\int_\mathbb{R}J_{\lambda}^{\varepsilon}w\;dx=
      \int_\mathbb{R}w\;dx}$
    for any $w\in\L1\left(\mathbb{R},\mathbb{R}\right)$,
    (conservation);
  \item
    $\left\|J^{\varepsilon}_{\lambda}w_{1}-J^{\varepsilon}_{\lambda}w_{2}
    \right\|_{\L1\left(\mathbb{R},\mathbb{R}\right)}\le
    \left\|w_{1}-w_{2}\right\|_{\L1\left(\mathbb{R},\mathbb{R}\right)}$,
    (contraction property);
  \item
    $\overline{\D\left(A^{\varepsilon}\right)}=
  \L1\left(\mathbb{R},\mathbb{R}\right)$ (density of the domain).
  \end{enumerate}
\end{theorem}

\begin{proof}
  The uniqueness, the monotonicity $(i)$ and the contraction property
  $(iii)$ 
  are direct consequences
  of Corollary~\ref{cor:contraction_eq}. 
  We now show the existence.
  For $\lambda>0$, $x\in\mathbb{R}$,
   we consider the traditional convolution kernel
  \begin{equation}
    \label{eq:conv_kernel}
    H\left(\lambda,x\right)=\frac{1}{2\sqrt{\lambda}}e^{-\frac{\left|x\right|}{\sqrt{\lambda}}}.
  \end{equation}
  It has the following properties:
  \begin{equation}
    \label{eq:conv_ker_prop}
        \begin{cases}
      H_{x}\left(\lambda,x\right)=-\frac{1}{\sqrt{\lambda}}\sign\left(x\right)H\left(\lambda,x\right),\\
      H_{xx}\left(\lambda,\cdot\right)=\frac{1}{\lambda}\left(H\left(\lambda,\cdot\right)-\delta_{0}\right),
      \\
      \lim_{\lambda\to 0^{+}}H\left(\lambda,\cdot\right)=\delta_{0},
      \quad \mbox{where $\delta_{0}$ is the unit mass at $x=0$,}\\
      \left\|H\left(\lambda,\cdot\right)\right\|_{\L1\left(\mathbb{R},\mathbb{R}\right)}=1,\qquad
      \left\|H_{x}\left(\lambda,\cdot\right)\right\|_{\L1\left(\mathbb{R},\mathbb{R}\right)}=\frac{1}{\sqrt{\lambda}}.
    \end{cases}
  \end{equation}
  
  Fix $\varepsilon,\lambda>0$ and for any $w\in\L1\left(\mathbb{R},\mathbb{R}\right)$ 
  we define the
  Lipschitz continuous map $\Lambda_{\lambda}^{w}:
  \L1\left(\mathbb{R},\mathbb{R}\right)\to\L1\left(\mathbb{R},\mathbb{R}\right)$
  as
  \begin{equation}
    \label{eq:Lambda_def}
    \left[\Lambda_{\lambda}^{w}\left(u\right)\right](x)=\int_{\mathbb{R}}H\left(\lambda
      \varepsilon,x-y\right)w(y)\; dy - \int_{\mathbb{R}}\lambda
    H_{x}\left(\lambda \varepsilon,x-y\right)f\left(y,u(y)\right)\; dy.
  \end{equation}
  
  By properties~\eqref{eq:conv_ker_prop}, it follows that
  $u\in\L1\left(\mathbb{R},\mathbb{R}\right)$ 
  is a weak solution
  to~\eqref{eq:res_eq} if and only if $u=\Lambda_{\lambda}^{w}\left(u\right)$.
  Moreover, one has 
  \begin{displaymath}
    \left\|\Lambda_{\lambda}^{w}\left(u_{1}\right)-\Lambda_{\lambda}^{w}\left(u_{2}\right)\right\|_{\L1\left(\mathbb{R},\mathbb{R}\right)}
    \le L\sqrt{\frac{\lambda}{\varepsilon}}
    \left\|u_{1}-u_{2}\right\|_{\L1\left(\mathbb{R},\mathbb{R}\right)},\quad
   \text{ for any }u_{1},u_{2}\in \L1\left(\mathbb{R},\mathbb{R}\right).
  \end{displaymath}
  
  Set $\lambda_{o}=\frac{\varepsilon}{2L^{2}}>0$, so that 
  $ L\sqrt{\frac{\lambda_{o}}{\varepsilon}}<1$. 
Then,  for any $\lambda\in]0,\lambda_{o}]$,
  $\Lambda_{\lambda}^{w}$ is a strict contraction in
  $\L1\left(\mathbb{R},\mathbb{R}\right)$.
  As a consequence, it has
  a unique fixed point $u\;=\;\Lambda_{\lambda}^{w}u$,  which we denote
  by $J_{\lambda}^{\varepsilon}w$. 
  We conclude that $\mathcal{R}\left(I+\lambda
    A^{\varepsilon}\right)=\L1\left(\mathbb{R},\mathbb{R}\right)$ holds
  for any $\lambda\in]0,\lambda_{o}]$, 
  and the domain of
  $A^{\varepsilon}$ is not empty. 
  
  Using the contraction property
  of $J_{\lambda_{o}}^{\varepsilon}$ and a
  classical argument that we repeat here for completeness (see
  \cite[Lemma~2.13]{Miyadera}), we prove that
  $\mathcal{R}\left(I+\lambda
    A^{\varepsilon}\right)=\L1\left(\mathbb{R},\mathbb{R}\right)$ for
  any $\lambda>0$. Indeed, fix $\lambda>\lambda_{o}$,
  $w\in\L1\left(\mathbb{R},\mathbb{R}\right)$, 
  we need to show that there is a function $u\in\D
  \left(A^{\varepsilon}\right)$
  that satisfies
  \begin{displaymath}
    \left(I+\lambda A^{\varepsilon}\right)u=w.
  \end{displaymath}
  Multiplying this equality by $\frac{\lambda_{o}}{\lambda}$, 
  algebraic manipulations give
  \begin{displaymath}
    \left(I+\lambda_{o}A^{\varepsilon}\right)u
    =\left(1-\frac{\lambda_{o}}{\lambda}\right)u+\frac{\lambda_{o}}{\lambda}w. 
  \end{displaymath}
  
  By the surjectivity of
  $\left(I+\lambda_{o}A^{\varepsilon}\right)$, the above equation is equivalent to the
  fixed point equation $u=T_{w}u$, where the map
  $T_{w}:\L1\left(\mathbb{R},\mathbb{R}\right)\to\L1\left(\mathbb{R},
    \mathbb{R}\right)$ is defined by
  \begin{displaymath}
    T_{w}u=J_{\lambda_{o}}^{\varepsilon}\left(\left(1-
        \frac{\lambda_{o}}{\lambda}\right)u+
      \frac{\lambda_{o}}{\lambda}w\right).
  \end{displaymath}
  Since $J_{\lambda_{o}}^{\varepsilon}$  is a contraction, we compute
  \begin{displaymath}
    \left\|T_{w}u_{1}-T_{w}u_{2}\right\|_{\L1\left(\mathbb{R},\mathbb{R}\right)}
    \le
    \left(1-\frac{\lambda_{o}}{\lambda}\right)\left\|u_{1}-
      u_{2}\right\|_{\L1\left(\mathbb{R},\mathbb{R}\right)}.
  \end{displaymath}
  One concludes that $T_{w}$ is a strict contraction in
  $\L1\left(\mathbb{R},\mathbb{R}\right)$, and hence it has a unique fixed
  point $u = T_{w} u$.
  
  To prove $(ii)$, it is enough to integrate over $\mathbb{R}$ the identity
  $u=\Lambda_{\lambda}^{w}u$ and apply Fubini's theorem.

  It remains to prove $(iv)$, the density of the domain of
  $A^{\varepsilon}$. Fix
  $w\in\L1\left(\mathbb{R},\mathbb{R}\right)$ and observe that
  $u_{\lambda}=J_{\lambda}^{\varepsilon}w\in\D\left(A^{\varepsilon}\right)$ for
  any $\lambda>0$. Hence it suffices to show that $u_{\lambda}\to w$
  as $\lambda\to 0$. 
  Fix a point $\bar  u\in\D\left(A^{\varepsilon}\right)$.  
  The contraction property of $J^{\varepsilon}_{\lambda}$  and~\eqref{7} imply
  \begin{equation}
    \begin{split}
      \left\|u_{\lambda}\right\|_{\L1\left(\mathbb{R},\mathbb{R}\right)}
      &\le\left\|J_{\lambda}^{\varepsilon}w-J_{\lambda}^{\varepsilon}\bar
        u \right\|_{\L1\left(\mathbb{R},\mathbb{R}\right)}+
      \left\|J_{\lambda}^{\varepsilon}\bar u-\bar
        u\right\|_{\L1\left(\mathbb{R},\mathbb{R}\right)} +\left\|\bar
        u\right\|_{\L1\left(\mathbb{R},\mathbb{R}\right)}\\
      &\le\left\|w-\bar u\right\|_{\L1\left(\mathbb{R},\mathbb{R}\right)}
      +\lambda\left\|A^{\varepsilon}\bar u\right\|_{\L1\left(\mathbb{R},\mathbb{R}\right)}
      +\left\|\bar u\right\|_{\L1\left(\mathbb{R},\mathbb{R}\right)}\\
      &\le C,
    \end{split}
  \end{equation}
  where $C$ is a constant independent of $\lambda\in]0,1]$.
  Using \textbf{f0)}, and the fact that $u_{\lambda}$ is the unique fixed point of
  $\Lambda_{\lambda}^{w}$, we compute
  \begin{equation}
    \begin{split}
      &\left\|u_{\lambda}-w\right\|_{\L1\left(\mathbb{R},\mathbb{R}\right)}=
      \left\|\Lambda_{\lambda}^{w}\left(u_{\lambda}\right)-
        w\right\|_{\L1\left(\mathbb{R},\mathbb{R}\right)}\\
      &\quad\le\int_{\mathbb{R}^{2}}H\left(\lambda\varepsilon,x-y\right)
      \left|w(x)-w(y)\right|\;dxdy+\int_{\mathbb{R}^{2}}
      \lambda \left|H_{x}\left(\lambda\varepsilon,x-y\right)\right|
      \left|f\left(y,u_{\lambda}(y)\right)\right|\; dxdy\\
      &\quad\le\int_{\mathbb{R}}H\left(\lambda\varepsilon,\xi\right)
      \left[\int_{\mathbb{R}}\left|w(x)-w(x-\xi)\right|\;dx\right]d\xi+\int_{\mathbb{R}}
      \frac{\lambda}{\sqrt{\lambda\varepsilon}}
      \left[\left|f\left(y,0\right)\right|+L\left|u_{\lambda}(y)\right|\right]\; dy\\
      &\quad\le\int_{\mathbb{R}}H\left(\lambda\varepsilon,\xi\right)
      \left[\int_{\mathbb{R}}\left|w(x)-w(x-\xi)\right|\;dx\right]d\xi+
      \sqrt{\frac{\lambda}{\varepsilon}}\left[L_{1}
      +L\cdot C\right]\\
    &\quad \overset{\lambda\to 0}{\longrightarrow}0,
  \end{split}
  \end{equation}
  completing the proof.
\end{proof}

We are now ready to apply Theorem~\ref{th:CL}.

\begin{theorem}
  \label{th:gen_linfty_f}
  If the flux $f$ satisfies the hypothesis \textbf{f0)},
  then the operator $A^{\varepsilon}$
  defined 
  in~\eqref{eq:Aepsilon_def} generates (in the sense of
  Theorem~\ref{th:CL}) a non linear continuous semigroup
  $S^{\varepsilon}_{t}:\L1\left(\mathbb{R},\mathbb{R}\right)
  \to\L1\left(\mathbb{R},\mathbb{R}\right)$
  of contractions.
  For any $\bar u\in\L1\left(\mathbb{R},\mathbb{R}\right)$,
  the trajectory of the semigroup
  $u(t,x)=\left(S^{\varepsilon}_{t}\bar u\right)(x)$
  belongs to
  $\C0\left([0,+\infty),\L1\left(\mathbb{R},\mathbb{R}\right)\right)$
  and
  is a weak solutions
  to the parabolic equation~\eqref{eq:vsc}. 
\end{theorem}

\begin{proof}
  Theorem~\ref{lem:has_sol0} guarantees that $A^{\varepsilon}$
  satisfies the hypotheses of Theorem~\ref{th:CL}. Therefore it
  generates a continuous semigroup
  $S_{t}^{\varepsilon}:
  \overline{\mathcal{D}\left(A^{\varepsilon}\right)}
  \to\overline{\mathcal{D}\left(A^{\varepsilon}\right)}$
  of contractions with
  $\overline{\mathcal{D}\left(A^{\varepsilon}\right)}=
  \L1\left(\mathbb{R},\mathbb{R}\right)$.
  For any $\bar u\in\L1\left(\mathbb{R},\mathbb{R}\right)$,
  the trajectory of the semigroup
  $u(t,x)=\left(S^{\varepsilon}_{t}\bar u\right)(x)$
  belongs to
  $\C0\left([0,+\infty),\L1\left(\mathbb{R},\mathbb{R}\right)\right)$.
  The trajectory $u$ can be obtained (see~\cite[Theorem~4.2]{Miyadera})
  as the limit in $\L1$ of approximations
  \begin{displaymath}
    u(t) = \lim_{\lambda\to 0}u_{\lambda}(t),\qquad 
    u_{\lambda}(t) =
    \left(J_{\lambda}^{\varepsilon}\right)^{\left[\frac{t}{\lambda}\right]}\bar u.
  \end{displaymath}
  
  By the definition of the resolvent
  $J_{\lambda}^{\varepsilon}$, 
  the approximations $u_{\lambda}(t)$ solve
  \begin{displaymath}
    \frac{u_{\lambda}(t,x)-u_{\lambda}\left(t-\lambda,x\right)}{\lambda}
    + \left[f(x,u_{\lambda}(t,x))-
      \varepsilon u_{\lambda,x}(t,x)\right]_{x}=0, \quad t\ge \lambda.
  \end{displaymath}
  We multiply this equation by a test function with compact
  support in $\left]0,+\infty\right[\times\mathbb{R}$, 
  and perform integrations by parts. 
  Taking the limit  $\lambda\to 0$, one shows that $u$ is a weak solution to
  the parabolic problem~\eqref{eq:vsc}.
\end{proof}

\section{The vanishing viscosity limit for the Backward Euler operator}
\setcounter{equation}{0}
\label{sec:vv}


In this section we study the
vanishing viscosity limit $\varepsilon\to 0$ in~\eqref{eq:vsc}, 
where we assume the hypotheses  \textbf{f1)}  on the flux $f$. 
Under \textbf{f1)}, the region $[0,1]$ is invariant for~\eqref{eq:res_eq}. 
We introduce the domain
\begin{equation}
  \label{eq:domain_def}
  D~\doteq~\left\{w\in \L1(\mathbb{R},\mathbb{R}):\ 0\le
    w\le 1\right\}.
\end{equation}
If the source term $w$ in \eqref{eq:res_eq} is in $D$, then 
$\underline u(x)=0$ and $ \overline u(x)=1$ are respectively
a lower and an upper solution to
\eqref{eq:res_eq}. 
An application of Theorem~\ref{thm:max_princ}
shows that $J_{\lambda}^{\varepsilon}w\in D$.

Hypothesis \textbf{f1)} implies additional regularity on the
solutions to~\eqref{eq:res_eq}.

\begin{lemma}
  \label{lem:equivalence}
  Suppose $f(x,\omega)$ satisfies \textbf{f1)}.
  If $w\in\Lloc1(\mathbb{R},\mathbb{R})$, then a function
  $u\in\Lloc1(\mathbb{R},\mathbb{R})$ is {\color{black} a weak solution} 
  to~\eqref{eq:res_eq} if and only if the following three conditions are
  satisfied:
  \begin{enumerate}[(i)]
  \item
    \label{item:1}
    $u\in \Wloc21 \left(\mathbb{R}\setminus\left\{0\right\},\mathbb{R}\right)
    \cap \Wloc11 \left(\mathbb{R},\mathbb{R}\right)$;
  \item
    \label{item:2}
    in $\mathbb{R}\setminus\left\{0\right\}$, $u$ is a weak (Sobolev) solution to
    \eqref{eq:res_eq};
  \item
    \label{item:3}
    the two limits $\lim_{x\to 0^{\pm}}u_{x}(x)=u_{x}(0\pm)$ exist and
    they satisfy
    \begin{equation}
      \label{eq:uxin0}
      f_{r}\left(u(0)\right) - f_{l}\left(u(0)\right)=\varepsilon
      \left(u_{x}(0+)-u_{x}(0-)\right).
    \end{equation}
  \end{enumerate}
  Moreover, we have
  \begin{equation}
    \label{eq:f-epsu_x}
    f(x,u) - \varepsilon
  u_{x}\in\Wloc11\left(\mathbb{R},\mathbb{R}\right).
  \end{equation}
\end{lemma}
\begin{proof}
  Suppose first that $u\in\Lloc1\left(\mathbb{R},\mathbb{R}\right)$
  is a weak solution to
  \eqref{eq:res_eq} in $\mathbb{R}$. Then it satisfies
  \begin{equation}
    \label{eq:firstinclusion}
    \lambda\left[f(x,u) - \varepsilon u_{x}\right]_{x} = w - u
    \in\Lloc1\left(\mathbb{R},\mathbb{R}\right)
    \; \Rightarrow f(x,u) - \varepsilon
  u_{x}\in\Wloc11\left(\mathbb{R},\mathbb{R}\right)
  \end{equation}
  proving \eqref{eq:f-epsu_x}.
  It further shows
  that $f(x,u) - \varepsilon
  u_{x}$ is continuous in $\mathbb{R}$.
  Since
    $f(x,u)\in\Lloc1\left(\mathbb{R},\mathbb{R}\right)$,
    \eqref{eq:firstinclusion} also implies
    $u_{x}\in\Lloc1\left(\mathbb{R},\mathbb{R}\right)$ and consequently
    $u\in\Wloc11\left(\mathbb{R},\mathbb{R}\right)$. 
    Therefore both
    $u$ and $f(x,u) - \varepsilon
    u_{x}$ are continuous in $\mathbb{R}$, and we have
  \begin{displaymath}
    \lim_{x\to 0-}\left[f(x,u) - \varepsilon
      u_{x}\right]=f_{l}\left(u(0)\right) - \varepsilon
      u_{x}(0-)=
    \lim_{x\to 0+}\left[f(x,u) - \varepsilon
      u_{x}\right]=f_{r}\left(u(0)\right) - \varepsilon
      u_{x}(0+), 
    \end{displaymath}
    concluding \eqref{item:3}. 
    
    Consider now the domain $\left]-\infty,0\right[$, where $f(x,u)=f_{l}(u)\in
    \Wloc11\left(\left]-\infty,0\right[,\mathbb{R}\right)$. 
     Then \eqref{eq:firstinclusion} imply $u_{x}\in
    \Wloc11\left(\left]-\infty,0\right[,\mathbb{R}\right)$ and
    hence $u\in
    \Wloc21\left(\left]-\infty,0\right[,\mathbb{R}\right)$. The same
    argument holds in the domain $\left]0,+\infty\right[$, proving
    \eqref{item:1}
    and \eqref{item:2}.

  Suppose now that $u\in\Lloc1\left(\mathbb{R},\mathbb{R}\right)$
  satisfies \eqref{item:1}, \eqref{item:2}, \eqref{item:3}.
  In
  $\mathbb{R}\setminus \left\{0\right\}$,
  equation \eqref{eq:res_eq} is equivalent to
  $\left[f(x,u) - \varepsilon u_{x}\right]_{x} = w - u$  and
  \eqref{item:3} implies that $f(x,u) - \varepsilon u_{x}$ is
    continuous at $x=0$.  Therefore $u$ is a weak solution to
    \eqref{eq:res_eq} on $\mathbb{R}$.
\end{proof}

\begin{corollary}
  \label{cor:smoothness}
    For $k\in\mathbb{N}$, if $w\in\C{k}(\mathbb{R},\mathbb{R})$,
    $f_{l},\;f_{r}\in\C{k+1}(\mathbb{R},\mathbb{R})$, and
  $u\in\Lloc1(\mathbb{R},\mathbb{R})$ is a weak solution to
  \eqref{eq:res_eq},
   then $u\in\C{k+2}\left(\left]-\infty,0\right[\cup\left]0,+\infty\right[,\mathbb{R}\right)$.
\end{corollary}
\begin{proof}
  In $\left]-\infty,0\right[$,
  we have that 
  $u_{xx}=\frac{1}{\varepsilon}\left(u + f_{l}'\left(u\right)u_{x}
    -w\right)$ holds.
 This relation, starting with the initial regularity
 given by Lemma~\ref{lem:equivalence} item \eqref{item:1}, by
 induction proves the result. The same holds in $\left]0,+\infty\right[$.  
\end{proof}

The next Lemma shows that the total variation of $J_{\lambda}^{\varepsilon}w$
is uniformly bounded with respect to the parameter $\varepsilon$.

\begin{lemma}
  \label{lem:unif_tv}
  Under the hypothesis \textbf{f1)}, the map
  $J_{\lambda}^{\varepsilon}$ defined in Theorem~\ref{lem:has_sol0}
  satisfies
  \begin{equation}
    \label{eq:unif_tv}
    \tv \left\{J_{\lambda}^{\varepsilon}w\right\}\le 2
    + \tv \left\{w\right\}, \quad \text { for all
    } w\in D.
  \end{equation}
\end{lemma}
\begin{proof}
  Consider first $w\in\Cc\infty\left(\mathbb{R},\mathbb{R}\right)$ and
  define $u=J_{\lambda}^{\varepsilon}w$. Lemma~\ref{lem:equivalence}
  and Corollary~\ref{cor:smoothness}
  imply that $u$ is smooth in
  $\mathbb{R}\setminus\{0\}$ and continuous in
  $\mathbb{R}$. 
  We claim the following: 
  \begin{itemize}
  \item If $\bar x\not=0$ is a point of local maximum for $u$, then 
  $u(\bar x) \le w(\bar x)$.
  \item If $\hat x \not =0$ is a point of local minimum for $u$, then $u(\hat x) \ge w(\hat x)$.
  \end{itemize} 
  Indeed, consider a local max $\bar x > 0$ (the case $\bar x < 0$ being completely similar).
  We have
  \[
    u_x(\bar x)=0, \quad u_{xx} (\bar x) \le 0,
  \]
  so 
  \[
  w(\bar x) - u(\bar x) = \lambda 
  \left[ (f_r )'(u(\bar x)) u_x(\bar x) - \varepsilon u_{xx}(\bar x)\right]
  \ge 0.
  \]

  Fix $\gamma<\tv\left\{u\right\}$ and points $x_{0}<x_{1}< \ldots
  <x_{J-1}<0<x_{J} <\ldots <x_{N}$ such that
  \begin{displaymath}
    \gamma <
    \sum_{i=1}^{J-1}\left|u\left(x_{i}\right)-u\left(x_{i-1}\right)\right|
    +\left|u(0)-u\left(x_{J-1}\right)\right|
    +\left|u(0)-u\left(x_{J}\right)\right|
    +\sum_{i=J+1}^{N}\left|u\left(x_{i}\right)-u\left(x_{i-1}\right)\right|.
  \end{displaymath}
  It is not restrictive to assume that
  $w\left(x_{0}\right)=w\left(x_{N}\right)=0$, and that the points $x_{i}$, for
  $1\le i \le J-1$, are alternatively points of local maximum and minimum
  for $u$
  beginning with a maximum at $x_{1}$ while, for $J\le i\le N-1$, 
  they are alternatively point of local maximum and minimum
  beginning with a maximum at $x_{N-1}$.
  Therefore we have
  $\left|u\left(x_{i}\right)-u\left(x_{i-1}\right)\right|\le\left|w\left(x_{i}\right)-w\left(x_{i-1}\right)\right|$
  for $0\le i\le N$, with $i\not= J$ which implies
  \begin{displaymath}
    \gamma <
    \sum_{i=1}^{J-1}\left|w\left(x_{i}\right)-w\left(x_{i-1}\right)\right|
    +1
    +1
    +\sum_{i=J+1}^{N}\left|w\left(x_{i}\right)-w\left(x_{i-1}\right)\right|\le
    \tv\left\{w\right\}+2.
  \end{displaymath}
  This proves the assertion because of the arbitrariness of $\gamma<\tv\left\{u\right\}$.

    Finally, given any $w\in D$ there exists a sequence
    $w_{\nu}\in\Cc\infty\left(\mathbb{R},
      \mathbb{R}\right)\cap D$ converging
    to $w$ in $\L1\left(\mathbb{R},\mathbb{R}\right)$ such that
    $\tv\left\{w_{\nu}\right\}
    \le \tv\left\{w\right\}$. The continuity of $J_{\lambda}^{\varepsilon}$
    and the lower semicontinuity of the total
    variation imply
    \begin{equation}
      \begin{split}
        \tv\left\{J_{\lambda}^{\varepsilon}w\right\} &\le
        \liminf_{\nu\to+\infty}
        \tv\left\{J_{\lambda}^{\varepsilon}w_{\nu}\right\} \le 2 +
        \liminf_{\nu\to+\infty}\tv\left\{w_{\nu}\right\} \\
        &\le 2
        +\tv\left\{w\right\}.
      \end{split}
    \end{equation}
    \end{proof}

The previous Lemma yields the compactness of the family
$\left\{J_{\lambda}^{\varepsilon}w\right\}_{\varepsilon>0}$
whenever $w$ has bounded total variation.
The limit is unique due to the following characterization.

\begin{theorem}
  \label{th:characterization}
  Given $w\in D\cap\bv\left(\mathbb{R},\mathbb{R}\right)$, from
  any sequence $\varepsilon_{\nu}\to 0$, we can extract a subsequence
  $\varepsilon_{\nu_{j}}$
  such that
  $J_{\lambda}^{\varepsilon_{\nu_{j}}}w$ converges pointwise to a function
  $u\in\bv\left(\mathbb{R},\mathbb{R}\right)$ which satisfies
  \begin{equation}
    \label{eq:limiteq}
    u + \lambda f(x,u)_{x}=w.
  \end{equation}
  Furthermore the
  following entropy inequality holds in the space of
  distributions
  \begin{equation}
    \label{eq:entropy_ineq}
    \lambda\delta_{0}\int_{0}^{u(0)}\eta''(\omega)\left[f_{r}(\omega)-f_{l}(\omega)\right]\;
    d\omega + \lambda q\left(x,u\right)_{x}
    +\eta'\left(u\right)\left[u-w\right]\le 0,
  \end{equation}
  where $\delta_{0}$ is the unit mass at the origin, $\eta$ is
  any smooth convex function, $q$ is defined by
  \begin{displaymath}
    q(x,\omega)=\int_{0}^{\omega}\eta'\left(\bar \omega\right)f_{\bar \omega}\left(x,\bar \omega\right)\;
    d\bar \omega.
  \end{displaymath}

  Moreover, if $u$ is discontinuous at $x_{o}$  with
  $u^{\pm}=u\left(x_{o}\pm\right)$, 
  we must have
    \begin{equation}
    \label{eq:conserv_cond}
    f\left(x_{o}-,u^{-}\right)=
    f\left(x_{o}+,u^{+}\right)\dot=\bar f,
  \end{equation}
  and  the following entropy conditions:
  \begin{enumerate}
  \item
    if $u^{-}<u^{+}$ and $x_{o}\not=0$ then
    \begin{displaymath}
      f\left(x_{o},k\right)\ge \bar f, \text{ for all } k\in\left[u^{-},u^{+}\right];
    \end{displaymath}
  \item
    if $u^{-}>u^{+}$ and $x_{o}\not=0$ then
    \begin{displaymath}
      f\left(x_{o},k\right)\le \bar f, \text{ for all } k\in\left[u^{+},u^{-}\right];
    \end{displaymath}
  \item
    if $u^{-}<u^{+}$ and $x_{o}=0$ then there exists
    $u^{*}\in\left[u^{-},u^{+}\right]$ such that
    \begin{displaymath}
      \begin{cases}
        f_{l}\left(k\right)\ge \bar f, \text{ for all }
        k\in\left[u^{-},u^{*}\right],\\
        f_{r}\left(k\right)\ge \bar f, \text{ for all }
        k\in\left[u^{*},u^{+}\right];
      \end{cases}
    \end{displaymath}
  \item
    if $u^{-}>u^{+}$ and $x_{o}=0$ then there exists
    $u^{*}\in\left[u^{+},u^{-}\right]$ such that
    \begin{displaymath}
      \begin{cases}
        f_{r}\left(k\right)\le \bar f, \text{ for all }
        k\in\left[u^{+},u^{*}\right],\\
        f_{l}\left(k\right)\le \bar f, \text{ for all }
        k\in\left[u^{*},u^{-}\right].
      \end{cases}
    \end{displaymath}
  \end{enumerate}
\end{theorem}
\begin{proof}
The proof takes a few steps.

\medskip

\noindent\textbf{Step 1.} 
  Define $u^{\varepsilon}=J_{\lambda}^{\varepsilon}w$. By
  Lemma~\ref{lem:unif_tv},  $\tv \{u^{\varepsilon}\}$ is bounded 
  uniformly in $\varepsilon$.
%
  Therefore there exists a subsequence $u^{\varepsilon_{\nu_{j}}}$
  which converges \emph{pointwise} to a function $u\in\bv\left(\mathbb{R},\mathbb{R}\right)$
  with $0\le u\le1$. To simplify the notation we denote $u^{\varepsilon}=u^{\varepsilon_{\nu_{j}}}$. 
  By definition of $J_{\lambda}^{\varepsilon}w$, $u^{\varepsilon}$ is a weak solution
  to~\eqref{eq:res_eq}.
  Passing to the limit as $\varepsilon\to 0$
  in~\eqref{eq:res_eq} we immediately obtain~\eqref{eq:limiteq}.
  
  \medskip
  
\noindent\textbf{Step 2.} 
  By Lemma~\ref{lem:equivalence}, given any smooth convex function
  $\eta\left(\xi\right)$, the composition
  $\eta\left(u^{\varepsilon}\right)$ is in
  $\Wloc11\left(\mathbb{R},\mathbb{R}\right)$  with
  $\eta\left(u^{\varepsilon}\right)_{x}=\eta'\left(u^{\varepsilon}\right)u^{\varepsilon}_{x}$. 
  Multiplying~\eqref{eq:res_eq} by the continuous function
  $\eta'\left(u^{\varepsilon}\right)$ we obtain
  \begin{equation}
    \label{eq:first_entropy}
      \eta'\left(u^{\varepsilon}\right)\lambda f\left(x,u^{\varepsilon}\right)_{x}- 
      \lambda\varepsilon \eta'\left(u^{\varepsilon}\right)u^{\varepsilon}_{xx} + \eta'\left(u^{\varepsilon}\right)\left[u^{\varepsilon}-w\right]=0.
    \end{equation}
    By Lemma~\ref{lem:equivalence},
    $u_{x}^{\varepsilon}\in\Wloc11\left(\mathbb{R}\setminus\left\{0\right\},\mathbb{R}\right)$ 
    with a possible discontinuity at $x=0$, 
   therefore $u_{x}^{\varepsilon}\in\bvloc\left(\mathbb{R},\mathbb{R}\right)$.
  Since $\eta'\left(u^{\varepsilon}\right)$ is locally Lipschitz
  we obtain by Leibniz
  rule (\cite[Proposition 3.2]{Ambrosiobv})
  \begin{displaymath}
    \left[\eta\left(u^{\varepsilon}\right)_{x}\right]_{x}=\left[\eta'\left(u^{\varepsilon}\right)u^{\varepsilon}_{x}\right]_{x}
    =\eta''\left(u^{\varepsilon}\right)\left(u^{\varepsilon}_{x}\right)^{2}
    +\eta'\left(u^{\varepsilon}\right)u^{\varepsilon}_{xx}.
  \end{displaymath}
  Using this equality, \eqref{eq:first_entropy} becomes
  \begin{displaymath}
      \eta'\left(u^{\varepsilon}\right)\lambda f\left(x,u^{\varepsilon}\right)_{x}-
      \lambda\varepsilon
      \eta\left(u^{\varepsilon}\right)_{xx} +
      \eta'\left(u^{\varepsilon}\right)\left[u^{\varepsilon}-w\right]=-\lambda\varepsilon
      \eta''\left(u^{\varepsilon}\right)\left(u^{\varepsilon}_{x}\right)^{2}
      \le 0.
  \end{displaymath}

  In the space of distribution,
  $\lambda\varepsilon\eta\left(u^{\varepsilon}\right)_{xx}\to 0$ and
  $\eta'\left(u^{\varepsilon}\right)\left[u^{\varepsilon}-w\right]\to
  \eta'\left(u\right)\left[u-w\right]$ as $\varepsilon\to 0$.
  It remains to show the weak convergence of the measure
  $\eta'\left(u^{\varepsilon}\right)f\left(x,u^{\varepsilon}\right)_{x}$. 
  
  We define the notations
  \begin{displaymath}
  q\left(x,\omega\right)=
    \begin{cases}
      q_{l}\left(\omega\right)&\text{ for }x\le 0,\\
      q_{r}\left(\omega\right)&\text{ for }x> 0,
    \end{cases}
      \end{displaymath}
      where
      \[
      q_{l}\left(\omega\right)=\int_{0}^{\omega}\eta'\left(\bar\omega\right)f_{l}'\left(\bar\omega\right)\;d\bar\omega,
      \qquad
      q_{r}\left(\omega\right)=\int_{0}^{\omega}\eta'\left(\bar\omega\right)f_{r}'\left(\bar\omega\right)\;d\bar\omega.
      \]
  
  Fix a test function $\varphi$. Observe that
  $\eta'\left(u^{\varepsilon}\right)f\left(x,u^{\varepsilon}\right)_{x}$
  has a Dirac mass at the origin.  We compute the duality product
  \begin{equation}
    \begin{split}
      \langle
      \eta' &\left(u^{\varepsilon}\right)  f\left(x,u^{\varepsilon}\right)_{x},\varphi\rangle\\
      &=\eta'\left(u^{\varepsilon}(0)\right)\left[f_{r}\left(u^{\varepsilon}\left(0\right)\right)
        -
        f_{l}\left(u^{\varepsilon}\left(0\right)\right)\right]\varphi(0)
     + \int_{-\infty}^{0}q_{l}\left(u^{\varepsilon}\right)_{x}\;
      \varphi\; dx +
      \int_{0}^{+\infty}q_{r}\left(u^{\varepsilon}\right)_{x}\;\varphi\;  dx\\
      &=
      \eta'\left(u^{\varepsilon}(0)\right)\left[f_{r}\left(u^{\varepsilon}\left(0\right)\right)
        -
        f_{l}\left(u^{\varepsilon}\left(0\right)\right)\right]\varphi(0)
      +\left[q_{l}\left(u^{\varepsilon}(0)\right)-q_{r}\left(u^{\varepsilon}(0)\right)\right]\varphi(0)\\
      &\qquad-\int_{-\infty}^{0}q_{l}\left(u^{\varepsilon}\right)
      \varphi_{x}\; dx -
      \int_{0}^{+\infty}q_{r}\left(u^{\varepsilon}\right)\varphi_{x}\;
      dx\\
      &=
      \Big\{\eta'\left(u^{\varepsilon}(0)\right)\left[f_{r}\left(u^{\varepsilon}\left(0\right)\right)
        -
        f_{l}\left(u^{\varepsilon}\left(0\right)\right)\right]+\left[q_{l}\left(u^{\varepsilon}(0)\right)-
        q_{r}\left(u^{\varepsilon}(0)\right)\right]\Big\}\varphi(0)\\
      &\qquad-\int_{\mathbb{R}}q\left(x,u^{\varepsilon}\right)\varphi_{x}\;
      dx.
    \end{split}
  \end{equation}
  Using   $f_{l}(0)=f_{r}(0)=0$ and integration by parts, we obtain
  \begin{equation}
    \begin{split}
      \eta'\left(u^{\varepsilon}(0)\right)&\left[f_{r}\left(u^{\varepsilon}\left(0\right)\right)
          -   f_{l}\left(u^{\varepsilon}\left(0\right)\right)\right]+
        \left[q_{l}\left(u^{\varepsilon}(0)\right)-q_{r}\left(u^{\varepsilon}(0)\right)\right]\\
      &=
      \eta'\left(u^{\varepsilon}(0)\right)\left[f_{r}\left(u^{\varepsilon}\left(0\right)\right)
        -   f_{l}\left(u^{\varepsilon}\left(0\right)\right)\right]+
      \int_{0}^{u^{\varepsilon}(0)}
        \eta'\left(\omega\right)\left[f_{l}'\left(\omega\right)-f_{r}'\left(\omega\right)\right]\;d\omega\\
      &= \int_{0}^{u^{\varepsilon}(0)}
        \eta''\left(\omega\right)\left[f_{r}\left(\omega\right)-f_{l}\left(\omega\right)\right]\;d\omega.
    \end{split}
  \end{equation}
This gives
  \begin{equation}
    \begin{split}
      \langle
      \eta'\left(u^{\varepsilon}\right)f\left(x,u^{\varepsilon}\right)_{x},\varphi\rangle
      &=\varphi\left(0\right)\int_{0}^{u^{\varepsilon}(0)}
      \eta''\left(\xi\right)\left[f_{r}\left(\xi\right)-f_{l}\left(\xi\right)\right]\;d\xi
      -\int_{\mathbb{R}}q\left(x,u^{\varepsilon}\right)\varphi_{x}\;
      dx\\
      &=\langle \delta_{0}\int_{0}^{u^{\varepsilon}(0)}
      \eta''\left(\xi\right)\left[f_{r}\left(\xi\right)-f_{l}\left(\xi\right)\right]\;d\xi
      + q\left(x,u^{\varepsilon}\right)_{x},\varphi\rangle.
    \end{split}
  \end{equation}
  Finally, since $u^{\varepsilon}$ converges pointwise to $u$ and is
  uniformly bounded, we have that the convergence
  \[\delta_{0}\int_{0}^{u^{\varepsilon}(0)}
      \eta''\left(\omega\right)\left[f_{r}\left(\omega\right)-f_{l}\left(\omega\right)\right]\;d\omega
      + q\left(x,u^{\varepsilon}\right)_{x}  \rightarrow
      \delta_{0}\int_{0}^{u(0)}
      \eta''\left(\omega\right)\left[f_{r}\left(\omega\right)-f_{l}\left(\omega\right)\right]\;d\omega
      + q\left(x,u\right)_{x}\]  
      holds in the space of the distributions, 
      completing the proof of~\eqref{eq:entropy_ineq}.

\medskip

\noindent\textbf{Step 3.} 
      The entropy conditions follow from the entropy
      inequality~\eqref{eq:entropy_ineq}. 
      Indeed, assume that $u$ has a jump at  $x_{o}$ 
      with      $u^{\pm}=u\left(x_{o}\pm\right)$. 
      Suppose $u^{-}<u^{+}$ while the other case being completely similar. 
      Since
      $\eta'\left(u\right)\left(u-w\right)$ is absolutely continuous
      with respect to the Lebesgue measure, computing the measure 
      of~\eqref{eq:entropy_ineq} at the singleton
      $\left\{x_{o}\right\}$ we obtain
      \begin{equation}
        \label{eq:jump_eq}
        \delta_{0}\left(\left\{x_{o}\right\}\right)\int_{0}^{u(0)}
        \eta''\left(\omega\right)\left[f_{r}\left(\omega\right)-f_{l}
          \left(\omega\right)\right]\;d\omega
      + q\left(x_{o}+,u^{+}\right)-q\left(x_{o}-,u^{-}\right)\le 0.
    \end{equation}
    
    For $k\in\left[0,1\right]$ and
    $i\in\mathbb{N}\setminus\left\{0\right\}$,
    we consider the following family of
    smooth convex functions
    $\eta_{k,i}$ and the corresponding fluxes $q_{k,i}$:
    \begin{displaymath}
      \eta_{k,i}\left(\omega\right)=\sqrt{\frac{1}{i}+\left(\omega-k\right)^{2}},\qquad
      q_{k,i}\left(x,\omega\right)=\int_{0}^{\omega}
      \eta'_{k,i}\left(\bar\omega\right)f_{\bar\omega}\left(x,\bar\omega\right)\; d\bar\omega.
    \end{displaymath}
    We have that, as $i\to+\infty$:
    \begin{equation}
      \begin{split}
      \eta_{k,i}\left(\omega\right)\to
      \left|\omega-k\right|&\qquad\text{ uniformly},\\
    \eta'_{k,i}\left(\omega\right)\to
    \sign(\omega-k)&\qquad\text{ pointwise},\\
    \eta''_{k,i}\left(\omega\right)\to
    2\delta_{k}&\qquad\text{ weakly$^*$ in the space of Radon measures},\\    
    q_{k,i}\left(x,\omega\right)\to
    \int_{0}^{\omega}\sign\left(\bar\omega-k\right)f_{\bar\omega}\left(x,\bar\omega\right)\;d\bar\omega
    &\qquad\text{ pointwise}.\\    
      \end{split}
    \end{equation}
    Here $\delta_{k}$ is the unit mass centered at $\omega=k$.
    
    We now substitute $\eta_{k,i}$ and $q_{k,i}$
    in~\eqref{eq:jump_eq} and take the limit as $i\to+\infty$. 
    We obtain, for any $k\not\in\left\{0,u(0)\right\}$:
      \begin{equation}
        \label{eq:jump_eq2}
        \begin{split}
          2\delta_{0}&\left(\left\{x_{o}\right\}\right)\chi_{[0,u(0)]}(k)
          \left[f_{r}\left(k\right)-f_{l}\left(k\right)\right]\\
          &+
          \int_{0}^{u^{+}}\sign\left(\omega-k\right)f_{\omega}\left(x_{o}+, \omega\right)\;d \omega
          -
          \int_{0}^{u^{-}}\sign\left(\omega-k\right)f_{\omega}\left(x_{o}-, \omega\right)\;d \omega
          \le 0.
        \end{split}
      \end{equation}
      Since the second and the third terms in the left hand side
      of~\eqref{eq:jump_eq2} are
      continuous with respect to $k$,
      it must hold for
      any $u(0),k\in[0,1]$.
      
      \medskip
      
\noindent\textbf{Step 4.}     Suppose $x_{o}>0$ (the case $x_{o}<0$ is completely similar).
    Then~\eqref{eq:jump_eq2} becomes (recall that we assume $u^{-}<u^{+}$)
    \begin{equation}
      \label{eq:xnotzero}
      \int_{u^{-}}^{u^{+}}\sign\left(\omega-k\right)f'_{r}\left(\omega\right)\;d\omega
      \le 0,\qquad \text{ for any } k\in\left[0,1\right].
    \end{equation}
    Evaluating~\eqref{eq:xnotzero} at $k=0$ gives 
            $f_{r}\left(u^{+}\right)-f_{r}\left(u^{-}\right)\le 0$, while
        at $k= 1$ it gives
        $f_{r}\left(u^{+}\right)-f_{r}\left(u^{-}\right)\ge 0$, 
        thus we conclude~\eqref{eq:conserv_cond}. Letting
        $k\in\left[u^{-},u^{+}\right]$,~\eqref{eq:xnotzero} becomes
        $\bar f\le f_{r}(k)$ which proves $1$.

        Finally we consider the case $x_{o}=0$ where
        $\delta_{0}\left(x_{o}\right)=1$. 
        Then, \eqref{eq:jump_eq2} becomes
      \begin{displaymath}
          2\chi_{[0,u(0)]}(k)
          \left[f_{r}\left(k\right)-f_{l}\left(k\right)\right]
          +
          \int_{0}^{u^{+}}\sign\left(\omega-k\right)f'_{r}\left( \omega\right)\;d \omega
          -
          \int_{0}^{u^{-}}\sign\left(\omega-k\right)f'_{l}\left( \omega\right)\;d \omega
          \le 0.
    \end{displaymath}
  %
    Setting  $k=0$ and $k=1$ in the above inequality  we obtain 
    \[f_{l}\left(u^{-}\right)=f_{r}\left(u^{+}\right)=\bar f,\]
    proving~\eqref{eq:conserv_cond}. 
    Then, with $k\in\left[u^{-},u^{+}\right]$ we
 get
 \begin{displaymath}
   \bar f \le 
          -\chi_{[0,u(0)]}(k)
          \left[f_{r}\left(k\right)-f_{l}\left(k\right)\right]
          +f_{r}\left(k\right)\qquad\text{ for any }k\in\left[u^{-},u^{+}\right].
    \end{displaymath}
    Letting
    \begin{displaymath}
    u^{*}=
 \begin{cases}
   u^{-} &\text{ if }u(0)\le u^{-},\\
   u(0) &\text{ if }u^{-}\le u(0)\le u^{+},\\
   u^{+} &\text{ if }u(0)\ge u^{+},\\   
 \end{cases}
\end{displaymath}
this proves 3.  This proof for 2.~and 4.~is completely similar.
\end{proof}

\bigskip

We now establish the uniqueness of the vanishing viscosity limit for 
backward Euler operator $J_{\lambda}^{\varepsilon}$.

\begin{theorem}
  \label{th:resolven_convergence}
  For any $w\in D$, $J_{\lambda}^{\varepsilon}w$ converges
  in $\L1\left(\mathbb{R},\mathbb{R}\right)$ to a unique limit
  $J_{\lambda} w\in D$ as $\varepsilon\to 0$, the map
  $J_{\lambda}:D\to D$ being a contraction.
\end{theorem}

\begin{proof}
  Suppose first $w\in \Cc\infty\left(\mathbb{R},\mathbb{R}\right)$,
  where the support of $w$ is contained in
  $\left[-M,M\right]$ for some $M>0$. 
  Define
  \[\overline
  u(x)=e^{\gamma\left(x+M\right)}.\]
  For $\gamma>0$
  sufficiently small independently of $\varepsilon\in\left(0,1\right)$, 
  we have, for all $x\in\left]-\infty,-M\right[$,
  \begin{displaymath}
  \overline u + \lambda\left[f\left(x,\overline u\right) -
  \varepsilon\overline u_{x}\right]_{x}=\overline u \left[1 +
    \lambda\gamma \left(f_{l}'\left(\overline u\right)-
    \varepsilon\gamma \right)\right]\ge 0=w.
\end{displaymath}
Therefore, for $\gamma>0$ small, in $\left]-\infty,-M\right[$,
$\overline u$
is an upper solution
to \eqref{eq:res_eq} satisfying
\begin{displaymath}
  \liminf_{x\to -\infty}\left[u(x)-\overline u(x)\right]=0,\quad
  \liminf_{x\to -M^{-}}\left[u(x)-\overline
    u(x)\right]=u\left(-M\right)-1\le 0,
\end{displaymath}
where $u(x)=\left(J_{\lambda}^{\varepsilon}w\right)(x)$.
Applying Theorem~\ref{thm:max_princ}, we have that
$0\le J_{\lambda}^{\varepsilon}w\le \overline u$ in $\left]-\infty,-M\right[$.
A similar argument holds in the interval
$\left]M,+\infty\right[$. 

Hence, for any
$\varepsilon\in \left(0,1\right)$ we have
\begin{equation}
  \label{eq:L1_bound}
  0\le J_{\lambda}^{\varepsilon}w \le
  \min\left\{1,e^{\gamma\left(x+M\right)},
    e^{-\gamma\left(x-M\right)}\right\}\in\L1\left(\mathbb{R},\mathbb{R}\right).
\end{equation}
By Theorem~\ref{th:characterization}, for any sequence
$\varepsilon_{\nu}\to 0$ there exists a subsequence
$\varepsilon_{\nu_{j}}$ such that
$u_{\nu_{j}}=J_{\lambda}^{\varepsilon_{\nu_{j}}}w$ converges pointwise
in $\mathbb{R}$ to a function $u$, and we have $0\le u\le 1$.
Using~\eqref{eq:L1_bound}, the dominated convergence theorem implies
that the pointwise limit $u$ is in
$\L1\left(\mathbb{R},\mathbb{R}\right)$ and that
$J_{\lambda}^{\varepsilon_{\nu_{j}}}w$ converges to $u$ in
$\L1\left(\mathbb{R},\mathbb{R}\right)$. The limit $u$ is 
{\color{black} a weak
solution to~\eqref{eq:limiteq}} 
and must satisfy all the properties in
Theorem~\ref{th:characterization}.  

We use contradiction to prove uniqueness of the limit. 
Assume that there are
two 
limit functions $u$ and $v$,
which satisfy all the properties of Theorem~\ref{th:characterization}. 
Since they have bounded total
variation, we consider their left continuous representatives. 
This
choice does not change at any point the left and right limits,
therefore~\eqref{eq:conserv_cond} and 1.~2.~3.~4. in 
Theorem~\ref{th:characterization} continue to
hold. Suppose that there exists a point $x_{o}$
such that $u\left(x_{o}\right)<v\left(x_{o}\right)$. Define (see
Figure~\ref{fig:claw15}):
\begin{equation}
  \begin{split}
    a &= \inf\left\{x\le x_{o}:\ u\left(\xi\right)< v\left(
        \xi\right) \text{ for any }\xi
      \in\left]x,x_{o}\right]\right\},\\
    b &= \sup\left\{x\ge x_{o}:\ u\left(\xi\right)\le v\left(
        \xi\right) \text{ for any }\xi
      \in\left[x_{o},x\right]\right\}.
  \end{split}
\end{equation}

    \begin{figure}[htbp]
    \centering
    \begin{tikzpicture}[xscale=0.9,yscale=0.9]
      \draw [line width=1.0,->] (-6,0)-- (6,0.);
      \draw [line width=1.0,dotted] (-6,4)-- (6,4.);
      \draw [line width=1.0,->] (0,-0.1)-- (0,5);
      \draw[color=black] (4.8,-0.3) node {$x$};
      \draw[color=black] (0.3,4.8) node {$u$};
      \draw[color=black] (-2.3,4.4) node {$u=1$};
      \draw[color=red] (-5.5,0.9) node {$v$};
      \draw[color=black] (-4.4,0.9) node {$u$};
      \draw[color=black] (-2.0,-0.4) node {$x_{o}$};
      \draw[color=black] (0,-0.4) node {$b$};
      \draw[color=black] (-3.0,-0.4) node {$a$};
      \draw[black,line width=0.2,dashed] (-2.0,0) -- (-2.0,2.3);
      \draw [black,line width=1.2] plot [smooth] coordinates {(-5,0.0) (-4,3)
        (-3,3.8) };
      \draw [black,line width=1.2] plot [smooth] coordinates {
        (-3,1) (-1,1) (0,0.2) };
      \draw [black,line width=1.2] plot [smooth] coordinates {
        (0,3) (1,3.5) (1.5,1.5) (2,0.3) (3,0.4) (4,2) (5,0)};
      \draw [black,line width=0.2,dashed] (-3,3.8) -- (-3,0);
      \draw [red,line width=1.2,dashed] plot [smooth] coordinates {(-6,0) (-5,1) (-4,3)
        (-3,3.1)};
      \draw [red,line width=1.2,dashed] plot [smooth] coordinates {
        (-3,2) (-1,2.5) (0,1) (0.5,0.5)
        (1.5,0.5) (2.5,0.4) (3.5,1.1) (4.2,1) (6,0)};
    \end{tikzpicture}%
    \caption{If there exists a point $x_{o}$ such that
      $v\left(x_{o}\right)>u\left(x_{o}\right)$ a contradiction is reached.}
    \label{fig:claw15}
  \end{figure}
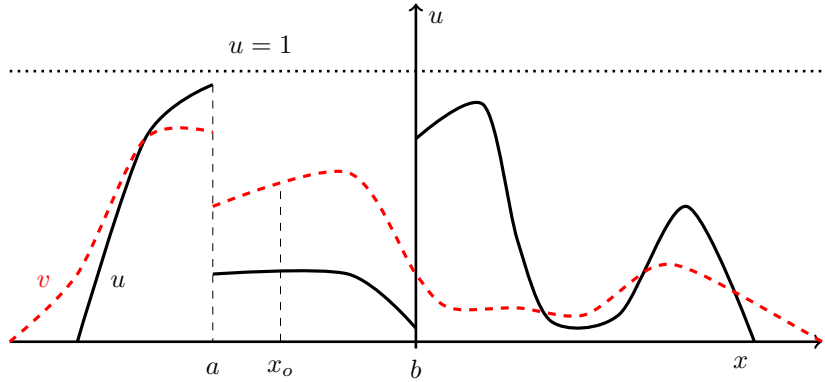

We have $a<x_{o}\le b$, 
and 
\begin{equation}
  \label{eq:definequalities}
  \begin{array}{ll}
    u(x)<v(x),\quad \text{for any }x\in ]a,x_{o}];
    &u(x)\le v(x),\quad \text{for any }x\in ]x_{o},b],\\
    v\left(a+\right)-u\left(a+\right)\ge 0,
    &v\left(b-\right)-u\left(b-\right)\ge 0,\\
    v\left(a-\right)-u\left(a-\right)\le 0, \text{ for }a\not=-\infty;
    &v\left(b+\right)-u\left(b+\right)\le 0, \text{ for }b\not= + \infty.
  \end{array}
\end{equation}

Observe that $f(x,u)$ and $f(x,v)$
are absolutely continuous thanks to~\eqref{eq:limiteq}.
Hence
integrating over 
the interval $\left]a,x_{o}\right[$ the identity 
\[v-u=
\lambda\left[f\left(x,u\right)-f\left(x,v\right)\right]_{x},\] 
we get
\begin{equation}
  \begin{split}
    0&<\int_{a}^{x_{o}}\left[v-u\right]\; dx \le\int_{a}^{b}\left[v-u\right]\; dx=
    \lambda\int_{a}^{b}\left[f(x,u)-f(x,v)\right]_{x}\; dx\\
    &=\lambda\left[f\left(b-,u(b-)\right)-f\left(b-,v(b-)\right)\right]-
    \lambda\left[f\left(a+,u(a+)\right)-f\left(a+,v(a+)\right)\right].
  \end{split}
\end{equation}

We claim that the entropy 
conditions of Theorem~\ref{th:characterization} imply 
\begin{equation}
  \label{eq:twoinequalities}
  f\left(b-,u(b-)\right)-f\left(b-,v(b-)\right)\le 0
  ~\mbox{and}~
  f\left(a+,u(a+)\right)-f\left(a+,v(a+)\right)\ge 0.
\end{equation}

The claim leads to the contradiction
\begin{displaymath}
  0<\int_{a}^{x_{o}}\left[v-u\right]\; dx\le 0.
\end{displaymath}
Thus, 
$J_{\lambda}^{\varepsilon}w$ converges in
$\L1\left(\mathbb{R},\mathbb{R}\right)$
to a unique limit $J_{\lambda} w$ that satisfies all the
properties of Theorem~\ref{th:characterization}.

Finally we take any function $w\in D$
and fix $\gamma>0$. 
Take $w_{\gamma}\in\Cc\infty\left(\mathbb{R},\mathbb{R}\right)\cap D$ such that
$\left\|w_{\gamma}-w\right\|_{\L1\left(\mathbb{R},\mathbb{R}\right)}<\gamma$. 
Then,
using the contraction property of $J_{\lambda}^{\varepsilon}$, we have
\begin{displaymath}
    \left\|J_{\lambda}^{\varepsilon}w-
      J_{\lambda}^{\mu}w
    \right\|_{\L1\left(\mathbb{R},\mathbb{R}\right)}\le
    2\gamma+
    \left\|J_{\lambda}^{\varepsilon}w_{\gamma}-
      J_{\lambda}^{\mu}w_{\gamma}
    \right\|_{\L1\left(\mathbb{R},\mathbb{R}\right)},
  \end{displaymath}
  so that
\begin{displaymath}
    \limsup_{\varepsilon,\mu\to 0}\left\|J_{\lambda}^{\varepsilon}w-
      J_{\lambda}^{\mu}w
    \right\|_{\L1\left(\mathbb{R},\mathbb{R}\right)}\le
    2\gamma.
  \end{displaymath}
  This proves that $J_{\lambda}^{\varepsilon}w$ is a
  Cauchy sequence in the complete metric space $D$, 
  hence it  converges in
  $D$ to a unique limit $J_{\lambda} w$. 
  Consequently $J_{\lambda}$ is also a contraction.

  It remains to prove the claim~\eqref{eq:twoinequalities}.
  We prove only the  first inequality, while the second being completely similar.
  If $b=+\infty$ then we have
  \begin{displaymath}
    f\left(b-,u(b-)\right)-f\left(b-,v(b-)\right)=f_{r}(0) - f_{r}(0)=0.
  \end{displaymath}
%
  If $b\in\left]0,+\infty\right[$ then
  \begin{displaymath}
    f\left(b-,u(b-)\right)-f\left(b-,v(b-)\right)=f_{r}(u(b-)) - f_{r}(v(b-)).
  \end{displaymath}
  Denote $u^{\pm}=u(b\pm), v^{\pm}=v(b\pm)$,  and suppose $u^{-}\le
  u^{+}$ (the other case being
  similar). By~\eqref{eq:conserv_cond} we have
  \begin{displaymath}
    \bar f_{u} \dot = f_{r}(u^{-}) = f_{r}(u^{+}),\quad 
    \bar f_{v} \dot = f_{r}(v^{-}) = f_{r}(v^{+}),\quad
    f_{r}(u(b-)) - f_{r}(v(b-))= \bar f_{u}-\bar f_{v}.
  \end{displaymath}
  If one of the two states $v^{-}$ or $v^{+}$ belongs to
  the interval $[u^{-},u^{+}]$, then 1.~in
  Theorem~\ref{th:characterization} implies $\bar f_{v}\ge \bar f_{u}$. 
  If none of them is on 
  the interval $[u^{-},u^{+}]$,   
  by~\eqref{eq:definequalities} we have $v^{+}\le u^{-}\le u^{+}\le  v^{-}$. 
  By point 2.~in
  Theorem~\ref{th:characterization} applied to the function $v$ we obtain 
  again $\bar f_{u}\le \bar f_{v}$, proving the claim.
  The case $b<0$ is completely similar.

Finally we consider the case $b=0$.  
  Again, suppose $u^{-}\le u^{+}$ and
  let $u^{*}$ be the state in point 3.~of Theorem~\ref{th:characterization}.
  If either $v^{-}\in  \left[u^{-},u^{*}\right]$ or
  $v^{+}\in  \left[u^{*},u^{+}\right]$  then 3.~in
  Theorem~\ref{th:characterization} implies $\bar f_{v}\ge \bar f_{u}$.
  If neither $v^{-}\in \left[u^{-},u^{*}\right]$ 
  nor $v^{+}\in\left[u^{*},u^{+}\right]$,  then
  by~\eqref{eq:definequalities} we have $v^{+}\le u^{*}\le v^{-}$.
  This relation, together with point 4.~in
  Theorem~\ref{th:characterization} applied to the function $v$ gives
  again $\bar f_{u}\le \bar f_{v}$ since there exists a point, namely
  $u^{*}\in \left[v^{+},v^{-}\right]$, such that
  $\bar f_{v}\ge \min\left\{f_{l}\left(u^{*}\right),f_{r}\left(u^{*}\right)\right\}\ge\bar
  f_{u}$.
  This completes the proof for the claim~\eqref{eq:twoinequalities}.
\end{proof}

\section{The vanishing viscosity limit for the
  generated semigroups}
\setcounter{equation}{0}
\label{sec:vvs}

In this section we apply Theorems~\ref{th:CL} and~\ref{th:BP} to approximate
the semigroup generated by
\begin{equation}
  \label{eq:9}
  u_{t}+f\left(x,u\right)_{x}=0
\end{equation}
with the semigroups generated
by the parabolic evolution equations
\begin{equation}
  \label{eq:10}
  u_t + f(x,u)_x~=~\varepsilon\, u_{xx},
\end{equation}
where the flux $f$ satisfies \textbf{f1)}.

    Theorem~\ref{th:resolven_convergence} implies
    \begin{equation}
      \label{eq:formal_resolv_conv}
      \lim_{\varepsilon\to 0}J_{\lambda}^{\varepsilon}w =
      J_{\lambda}w,\quad\text{ for any }\quad w \in D,\; \lambda>0,
    \end{equation}
    $J_{\lambda}$ being a family of contractions in
    $D$ with $D$ defined in~\eqref{eq:domain_def}. 
    Therefore we can define
    the (possibly multivalued) map
    \begin{equation}
      \label{eq:A_definition}
      A = \left\{\left(J_{\lambda}w,\frac{1}{\lambda}\left(w-J_{\lambda}w\right)\right):\ w
        \in D,\; \lambda>0\right\}\subset D
      \times \L{1}\left(\mathbb{R},\mathbb{R}\right).
    \end{equation}
    \begin{remark}
      \label{rem:domains}
      Recalling~\eqref{eq:domDef}, the domain
      of $A$ is given by:
      \begin{equation}
        \label{eq:A_domain}
        \mathcal{D}\left(A\right)=\left\{u\in D:\ \text{
            there exist }w\in D,\; \lambda>0 \text{ such that }u=J_{\lambda}w\right\}.
      \end{equation}
      Therefore, if $u\in\mathcal{D}\left(A\right)$, then, for some
      $w\in\mathcal{D}$, $\lambda>0$: 
      $u=J_{\lambda}w=\lim_{\varepsilon\to 0}u^{\varepsilon}$ with
      $u^{\varepsilon}=J_{\lambda}^{\varepsilon}w$. Since $J_{\lambda}^{\varepsilon}$ is the
      resolvent of $A^{\varepsilon}$, the function
      $u^{\varepsilon}$ solves $u^{\varepsilon}+
      \lambda\left[f\left(x,u^{\varepsilon}\right)-\varepsilon
        u^{\varepsilon}_{x}\right]_{x}=w$.
      Taking the weak limit of this equation as $\varepsilon\to 0$ we
      have that $u$ is a weak solution to $u+
      \lambda f\left(x,u\right)_{x}=w$.  
      This implies
      $w-J_{\lambda}w=w-u=\lambda f\left(x,u\right)_{x}$, therefore the operator $A$
      is single valued with
      \begin{equation}
        \label{eq:A_values}
      Au=f\left(x,u\right)_{x}\in
      \L{1}\left(\mathbb{R},\mathbb{R}\right) \text{ for any }
      u\in\mathcal{D}\left(A\right).
    \end{equation}
    The operator $A$ is a candidate as a generator for the evolution
    equation~\eqref{eq:9}.

    Take now
    $u\in
    D\cap\Cc\infty\left(\mathbb{R}\setminus\left\{0\right\},\mathbb{R}\right)$,
    satisfying $u(x)<1$ for any $x\in\mathbb{R}$.
    Then, for a suitably small $\lambda>0$ and all $\varepsilon\in]0,1[$, we have
    \begin{displaymath}
    w^{\varepsilon}=u + \lambda\left[f(x,u)-\varepsilon u
      _{x}\right]_{x}\in
    \Cc\infty\left(\mathbb{R}\setminus\left\{0\right\},\mathbb{R}\right)\cap
    D.
  \end{displaymath}
    Moreover
    $w^{\varepsilon}$ converges to $w=u+\lambda f(x,u)_{x}\in D$ in
    $\L{1}\left(\mathbb{R},\mathbb{R}\right)$ as $\varepsilon\to 0$.
    Since $u=J_{\lambda}^{\varepsilon}w^{\varepsilon}$ we compute
    \begin{displaymath}
      \left\|u-J_{\lambda}w\right\|_{\L{1}\left(\mathbb{R},\mathbb{R}\right)}
      =\lim_{\varepsilon \to 0}
      \left\|J_{\lambda}^{\varepsilon}w^{\varepsilon}-
        J_{\lambda}^{\varepsilon}w\right\|_{\L{1}\left(\mathbb{R},\mathbb{R}\right)}
      \le \lim_{\varepsilon \to 0}
      \left\|w^{\varepsilon}-w\right\|_{\L{1}\left(\mathbb{R},\mathbb{R}\right)}
      =0.
    \end{displaymath}
    This means that $u=J_{\lambda}w$ and hence
    $u\in\D\left(A\right)$. This implies 
    $\overline{\D\left(A\right)}=D$ i.e. the domain
    of $A$ is dense in $D$.
  \end{remark}

  \begin{theorem}
    \label{th:BP_limit}
      The map $A$ defined in~\eqref{eq:A_definition} (or,
      alternatively in~\eqref{eq:A_domain}, \eqref{eq:A_values})
      generates a unique
    continuous semigroup
    of contractions $S_{t}:D\to D$ whose
    trajectories are weak solutions to~\eqref{eq:9}.
    Moreover,  
    let $S^{\varepsilon}_{t}:D \to D$ be
    the semigroup generated by $A^{\varepsilon}$ in~\eqref{eq:Aepsilon_def}, then
    the following limit holds
    \begin{equation}
      \label{eq:wl}
      S_{t} \bar u~=~\lim_{\varepsilon\to 0}
      S^{\varepsilon}_{t}\bar u,\quad
      \text{ for all }\bar u \in D\text{ uniformly on
        bounded $t$ intervals}.
    \end{equation}
\end{theorem}
\begin{proof}
  Take any $w\in D$, $\lambda>0$, by
  definition~\eqref{eq:A_domain},
  $J_{\lambda}w\in\mathcal{D}\left(A\right)$, therefore,
  using~\eqref{eq:A_definition} we compute
  \begin{displaymath}
    \left(I + \lambda A\right)J_{\lambda}w=J_{\lambda}w + \lambda
    \frac{1}{\lambda}\left(w- J_{\lambda}w\right)=w.
  \end{displaymath}
  Therefore, for any $\lambda>0$, we have
  $\overline{\mathcal{D}\left(A\right)}=D\subset\mathcal{R}\left(I+\lambda A\right)$.
  The previous equality also shows that the resolvent of $A$ is
  $J_{\lambda}$ which is a contraction on $D$.
  By the Crandall \& Liggett generation theorem, Theorem~\ref{th:CL}, the map $A$
  generate a semigroup of contractions
  $S_{t}$ defined on $D$.
  
  We are now in a position to apply the result by Brezis
  and Pazy.
  Since
    $\lim_{\varepsilon\to 0}J_{\lambda}^{\varepsilon}w = J_{\lambda}w$
    for any $w\in D=\overline{\mathcal{D}\left(A\right)}$ and
    $\lambda>0$ with $\mathcal{D}\left(A\right)\subset
    \overline{\mathcal{D}\left(A^{\varepsilon}\right)}=
    \L{1}\left(\mathbb{R},\mathbb{R}\right)$, 
      Theorem~\ref{th:BP} implies~\eqref{eq:wl} for the
      corresponding semigroups.
      Finally, passing to the
    limit in the weak formulation for~\eqref{eq:10},  the
    trajectories of $S_{t}$ are weak solutions to~\eqref{eq:9}.
  \end{proof}

\section{Counter examples on adapted entropies and their applications} \label{sec:Counter}

We first 
observe that the entropy solutions selected by the adapted entropies
  approach~\cite{AP1,BaiJen,Gwiazda1,CheEveKli,Panov}
  are not, in general, the vanishing viscosity limits. 
  We consider the example given in~\cite[Section 5]{AP1}, the paper where
  the concept was originally introduced. 
  One considers a conservation law  
  \begin{equation}\label{vv0}
  u_t + f(x,u)_x=0, \qquad  \mbox{where} \quad
  f(x,u) 
  = \begin{cases} \displaystyle f_l(u)= \frac12 (u-1)^2 , & x\le0,\\[2mm]
 \displaystyle f_r(u) = \frac12 {u^2}, & x>0. \end{cases} 
  \end{equation}
  See Figure~\ref{fig:RP} for an illustration of 
 the graphs for the flux functions $f_l$ and $f_r$.

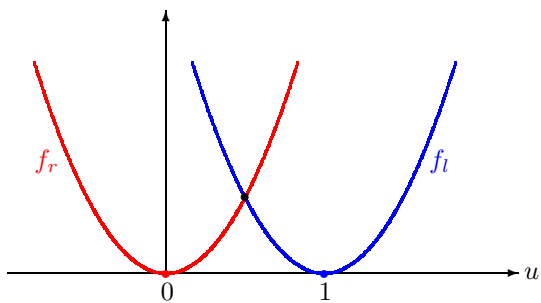
\begin{figure}[htbp]
\begin{center}
\setlength{\unitlength}{0.7mm}
\begin{picture}(90,55)(0,-5)  
\put(0,0){\vector(1,0){97}}\put(98,-1){$u$}
\put(30,0){\vector(0,1){50}}
\put(29,-5){$0$}\put(59,-5){$1$}
\thicklines
\color{red}
\qbezier(5,40)(30,-40)(55,40)\put(5,20){$f_r$}
\put(30,0){\circle*{1.5}}
\color{blue}
\qbezier(35,40)(60,-40)(85,40)\put(80,20){$f_l$}
\put(60,0){\circle*{1.5}}
\color{black}
\put(45,14.5){\circle*{1.5}}
\end{picture}
\caption{Graphs for the flux functions $f_l$ and $f_r$.}
\label{fig:RP}
\end{center}
\end{figure}

Below we give several examples, on various cases and aspects of this problem.

\begin{example}\label{ex:C1}
In this example we show that the solution that satisfies the adapted entropy 
is different from the one obtained by vanishing viscosity. 
Consider the Riemann problem
  \begin{equation}\label{eq:iii1}
  u_t +f(x,u)_x=0, \qquad 
  u(0,x)=\begin{cases} u_l =\frac{1}{2}, & x\le0,\\
  u_r = \frac12, & x>0, \end{cases} 
  \end{equation}
  and the corresponding viscous equation
   \begin{equation}\label{eq:iii1v}
     u_{t}+f(x,u)_{x}=\varepsilon u_{xx}, \qquad       u(0,x)=\frac{1}{2}.
\end{equation}
  
 We observe that the graphs of $f_l$ and $f_r$ intersect at $u=\frac12$ where
 $ f_l(\frac12) = f_r(\frac12) $.
Therefore the  constant function $u^{\varepsilon}(t,x)=\frac{1}{2}$ is the solution
of the Cauchy problem for the viscous equation~\eqref{eq:iii1v}
  for any $\varepsilon>0$. 
  Hence as $\varepsilon\to 0+$, the solution 
  $u^{\varepsilon}$ converges strongly to the constant function $u(t,x)=\frac12$, 
  which is the  vanishing  viscosity solution for the non-viscous equation 
 with the same initial data~\eqref{eq:iii1}. 

  However, the solution selected by the adapted entropy with initial condition~\eqref{eq:iii1} 
  is different. 
  From formula~\cite[(5.7)]{AP1}, we see that 
  the adapted entropy solution consists of three parts:
  \begin{itemize}
 \item  
 a  rarefaction wave with negative characteristic speed, solving the
  Riemann problem 
  \[ u_t + f_l(u)_x =0, \qquad u(0,x)=\begin{cases} \frac12, & x<0 ,\\ 1 , & x>0,\end{cases} \]
  \item 
  a discontinuity at $x=0$, with   the traces   $u(t,0-)=1$ and $ u(t, 0+)  =0$, and
 
  \item 
  a  rarefaction wave with positive characteristic speed, solving the
  Riemann problem 
  \[ u_t + f_r(u)_x =0, \qquad u(0,x)=\begin{cases} 0, & x<0 ,\\ \frac12 , & x>0.\end{cases} \]
  \end{itemize} 
\end{example}

\begin{example}\label{ex:C2}   
We now show that, a discontinuity satisfying the adapted entropy condition 
can not be obtained as the vanishing viscosity limit of a viscous traveling wave. 
  Formula~\cite[(5.7)]{AP1} further implies that, for the Cauchy problem 
  with the initial condition
  \begin{equation}\label{vv6}
  u(0,x) = \begin{cases}  1, & x<0 ,\\ 0, & x>0,\end{cases}  
  \end{equation}
  the adapted entropy solution is stationary  in time, i.e.  
  \begin{equation}\label{vv5}
   u(t,x)=u(0,x)
   = \begin{cases}  1, & x<0 ,\\ 0, & x>0,\end{cases}  , \qquad \forall t \ge 0. 
   \end{equation}
   
  We claim that the  solution~\eqref{vv5} can not be obtained by vanishing viscosity 
  of a viscous traveling wave for the viscous equation
  \begin{equation}\label{vv}
  u_t + f(x,u)_x = \varepsilon u_{xx}.
  \end{equation}
  
  Indeed, fix $\varepsilon>0$, and let $U$ be a monotone stationary viscous profile for~\eqref{vv}, 
  satisfying the asymptotic limits
    \begin{equation}\label{vv2}
  \lim_{x\to-\infty} U(x) = 1, \qquad \lim_{x\to+\infty} U(x) = 0.
  \end{equation}
  Then, $U$ satisfies the ODE
      \begin{equation}\label{vv3}
  \varepsilon U'=
  f(x,U).
  \end{equation}
  
  With the asymptotic conditions~\eqref{vv2} we seek monotonically decreasing 
  solutions,   i.e. $U'(x)\le 0$  for all $x\in\mathbb{R}$. 
  However from Figure~\ref{fig:RP} it is clear that, for every $U$ between $0$ and $1$, we have
  both $f_l(U) >0$ and $f_r(U)>0$, therefore $f(x,U)>0$ and thus 
  $U'(x) >0$ for all $x\in\mathbb{R}$ and $U\in(0,1)$, 
  a contradiction.
  We conclude that no stationary,
  monotonically decreasing viscous wave profiles can exist with the asymptotic 
  conditions~\eqref{vv2}, proving the claim.
  \end{example} 
  
  \begin{remark}
  We remark that a vanishing viscosity Riemann solver was 
  constructed in~\cite{GS2016} and a rigorous proof was given. 
Following the algorithm in~\cite{GS2016},
the unique vanishing viscosity solution for the Riemann problem in Example~\ref{ex:C2}
consists of (i) a shock from $u=1$ to $u=\frac12$ with negative wave speed for $x<0$, 
(ii) $u(t,0-)=u(t,0+)=\frac12$ at $x=0$, and 
(iii) a shock from $u=\frac12$ to $u=0$ with positive wave speed for $x>0$.
 \end{remark}

\begin{example}\label{ex:C3}
It would be of interest to analyze rigorously the admissible 
solutions selected by the adapted entropies
as limit solutions of some regularization (for example vanishing viscosity), 
a problem which is still open. 
Preliminarily, 
by comparing the above examples to Example 4.4 in~\cite{MR3663611}, it appears that 
the adapted entropy condition selects the solution with $\kappa=\infty$,
i.e., with $\varepsilon_n\equiv 0$ in~\eqref{eq:approx2}.
We now provide a simple proof for this claim in the setting of this example.

The flux in~\eqref{vv0} can also be rewritten as
\[
f(x,u) = \frac{(u-1+H(x))^2}{2},
\]
where $H$ is the Heaviside step function. 
Consider a  smooth and monotone mollification $H^\delta$ such that
\[
H^\delta(x) =  \begin{cases} 0 ~ &\mbox{for} ~x\le -\delta,\\
  1 ~ &\mbox{for} ~x\ge \delta, \end{cases} 
  \qquad
(H^\delta)'(x) \ge 0, ~\mbox{for} ~~ |x| \le \delta,
\]
and
\[
\lim_{\delta\to 0} H^\delta(x) = H(x)\qquad \mbox{pointwise} ~\forall x\in \mathbb{R}\setminus \{0\} .
\]
We denote the mollified flux as 
\[
f^\delta(x,u) \;\dot=\; \frac{(u-1 + H^\delta(x) )^2}{2}.
\]

Fix an $x$, we have
\[ 
(f^\delta)_u(x,u) = u-1+H^\delta(x) . 
\]
Therefore the minimum of the mapping $u\mapsto f^\delta$ is 
at 
\[ u_m(x)=1-H^\delta(x), \qquad \mbox{where}\quad 
f^\delta(x,u_m(x)) =0\quad \forall x\in\mathbb{R}.
\]
One can readily verify that the smooth function $u^\delta$ defined as
\[
u^\delta(t,x) = \begin{cases}  1, \quad & x\le -\delta,\\
u_m(x) = 1 -H^\delta(x), \qquad & |x| \le \delta,\\
0, & x\ge \delta,
\end{cases}
\]
is a stationary solution of the Cauchy problem for the conservation law
\[
u_t + f^\delta(x,u)_x =0
\]
with initial condition $u^\delta(0,\cdot)$.   
Taking the limit $\delta \to 0$, we see  that
$u^\delta(t,\cdot) $  converges to $u(t,\cdot)$  in~\eqref{vv5},
which is the solution selected by the adapted entropies. 
This proves our claim.
\end{example}  

The analysis in Example~\ref{ex:C3} applies only to this specific example, 
with the specific choice of mollification. 
A rigorous analysis for the general cases is beyond the scope of this paper, 
and could be the topic of a separated future work.

\begin{remark}
We remark that the adapted entropies require 
strong restrictions on the fluxes, even in  one space dimension.
In~\cite{AP1} (H$3'$), the flux can have at most one
single minimum (or maximum) \emph{at the same level} for any $x$.
This restricts a direct application to models of traffic flow with rough 
road conditions and road junctions (see~\cite{GNPT1}), 
where a typical flux function is
$f(x,u) = V(x) u(1-u)$ with $V$ discontinuous. 
In contrast, our hypothesis allows the
presence in the flux both at $x>0$ and $x<0$ of any number of
maxima/minima at any number of different levels.
\end{remark}

The adapted entropy concept is utilized in~\cite{Panov} to
establish uniqueness of solutions for scalar conservation laws with  discontinuous flux. 
The procedure introduced in~\cite[(1.5)]{Panov}
allows the study of rather general right and left flux functions,
in the adapted entropies framework. 
Unfortunately, in the case where the right and left fluxes have extrema at
different levels, one obtains non-physical solutions in applications. 
Below we give a concrete example. 

 \begin{example}
   \label{exa:adapted2}
Consider the Riemann problem for traffic flow
   \begin{displaymath}
     \begin{cases}
       u_{t}+f\left(x,u\right)_{x}=0,\\
       u\left(0,x\right)=\frac{1}{2},
     \end{cases}
   \end{displaymath}
   for the flux
   \begin{displaymath}
     f\left(x,\omega\right)=
     \begin{cases}
       f_{l}\left(\omega\right)=8\omega\left(1-\omega\right)&\text{ if
       }x\le0\text{ and }\omega\in\left[0,1\right],\\
       f_{r}\left(\omega\right)=4\omega\left(1-\omega\right)&\text{ if
       }x>0\text{ and }\omega\in\left[0,1\right].
     \end{cases}
   \end{displaymath}
   Depending on the choice of $g$ and $\beta$ satisfying
   $f(x,\omega)= g\left(\beta\left(x,\omega\right)\right)$ as
   in~\cite{Panov} one can get only two types of solutions. The first
   one is a connection in the sense of~\cite{MisVee,BurKarKenTow} with
   $A=1$ and $B=0$ as it has already been observed in~\cite[after
   (4.3)]{Panov}. The other type of solutions take values outside the
   interval $[0,1]$ which is ``nonphysical'' since the conserved
   variable is a density function. 
   Other types of solutions studied in
   the literature, such as the ones obtained with connections using
   different $A$ and $B$ values, or the one obtained by vanishing
   viscosity, do not satisfy the adapted entropy condition. 
 \end{example}

We finally remark that the solutions considered in~\cite{Panov} are not the
vanishing viscosity solutions in general case, 
from the discussions in Examples~\ref{ex:C1}-\ref{ex:C3}.

  \section{Examples and concluding remarks}
\begin{example}
We first give several examples of the backward Euler operators for the 
non viscous conservation law
\[ u_t + f(u)_x=0.\]
If $f_u$ has a fixed sign, say $f_u >0$, then 
for any $\lambda>0$
the backward Euler operator $J_\lambda$ generates a continuous function $u$,  
even for discontinuous function of $w$. 
In this simpler case,  the entropy condition is automatically satisfied,
and the operator generates a Lipschitz semigroup of 
entropy weak solution for the conservation law \cite{C72}.
However, when $f_u$ changes sign, the backward Euler solution might not be unique,
and entropy conditions
(such as in Theorem~\ref{th:characterization}) 
are required to single out the admissible solution.

To fix the idea, we consider the traffic flow model with $f(u)=u(1-u)$. 
Given $w$, the solution $u=J_\lambda w$ satisfies the ODE
\begin{equation}\label{BEode}
u'(x) =\frac{w-u}{\lambda f'(u)} = \frac{w-u}{2 \lambda (0.5-u)} .
\end{equation}
If $w(x)$ is piecewise constant, the solution for the above ODE can be constructed explicitly
on each interval where $w(x)$ is constant.
One can then piece them together to form a solution on the whole real line. 

We observe that $u'$ blows up at $u=0.5$, unless $w=0.5$ also. 
When $w=0.5$, we have $u'=1/(2\lambda)$ if $u\not=0.5$.
Since $u'$ can be anything at $u=0.5$, we also have $u\equiv0.5$ as a solution.

We consider 3 typical cases, where we use $\lambda=0.5$.

\textbf{Case 1.} If we use the initial condition
\[
w(x) = \begin{cases} 0.5 & (|x|>1), \\
1 & (-1 <x<1),
\end{cases} 
\]
the solution of the conservation law consists of a shock 
at $x=-1$ and a rarefaction at $x=1$. 
Furthermore, we have $f'(u) \le0$ in the solution. 
Consequently, the backward Euler solution $u(x)$ contains no jump, 
see Figure~\ref{fig:Case1} for a qualitative illustration.
However, we observe vertical tangent at $x=1$ where $u=0.5$ and $f'(u)=0$. 

\begin{figure}[htbp]
\begin{center}
\setlength{\unitlength}{0.8mm}
\begin{picture}(120,48)(0,0)
\put(0,0){\vector(1,0){120}}
\multiput(0,20)(1,0){120}{\line(1,0){0.5}}
\multiput(40,0)(0,1){40}{\line(0,1){0.5}}
\multiput(80,0)(0,1){40}{\line(0,1){0.5}}
\put(40,38){\line(-1,-1){38}}\put(40,30){\line(-1,-1){30}}\put(40,20){\line(-1,-1){20}}
\put(40,10){\line(-1,-1){10}}\put(32,40){\line(-1,-1){30}}\put(22,40){\line(-1,-1){20}}
\put(80,0){\line(1,1){40}}\put(90,0){\line(1,1){30}}\put(100,0){\line(1,1){20}}
\put(80,10){\line(1,1){30}}\put(80,20){\line(1,1){20}}\put(80,30){\line(1,1){10}}
\qbezier(40,38)(108,28)(60,0)
\qbezier(40,32)(100,25.5)(50,0)
\qbezier(40,27)(90,24)(40,3)
\linethickness{0.45mm}
\put(0,20){\line(1,0){40}}\put(40,20){\line(0,1){20}}\put(40,40){\line(1,0){40}}
\put(80,40){\line(0,-1){20}}\put(80,20){\line(1,0){40}}
\put(81,28){$w$}
\color{red}
\put(0,20){\line(1,0){22}}
\qbezier(22,20)(31,29)(40,38)
\qbezier(40,38)(80,32)(80,20)
\put(80,20){\line(1,0){40}}
\put(94,21){$u$}
\end{picture}
\caption{Case 1. Plots of $w(x)$ (thick black lines), possible solutions of~\eqref{BEode} (thin black curves), and $u(x)$ (thick red curves) for the case with $f'(u)\le 0$.  Here $u(x)$ contains no jumps.}
\label{fig:Case1}
\end{center}
\end{figure}
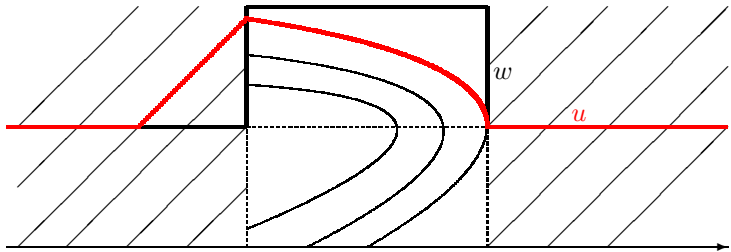

\textbf{Case 2.} If we use the initial condition
\[w(x) = \begin{cases} 0.5 & (x<-1), \\
1 & (-1 <x<0),\\
0.25 & (0<x<0.5),\\
0.5 & (x>0.5),
\end{cases}
\]
the solution for the conservation law consists of a transonic rarefaction 
initiated at $x=0$.  
The backward Euler solution $u(x)$ contains no jumps, 
see Figure~\ref{fig:Case2},
although  the gradient is infinite at $x=0$ where $u=0.5$.

\begin{figure}[htbp]
\begin{center}
\setlength{\unitlength}{0.8mm}
\begin{picture}(140,48)(0,0)
\put(0,0){\vector(1,0){140}}
\multiput(0,20)(1,0){140}{\line(1,0){0.5}}
\multiput(40,0)(0,1){40}{\line(0,1){0.5}}
\multiput(80,0)(0,1){40}{\line(0,1){0.5}}
\multiput(100,0)(0,1){40}{\line(0,1){0.5}}
\put(40,38){\line(-1,-1){38}}\put(40,30){\line(-1,-1){30}}\put(40,20){\line(-1,-1){20}}
\put(40,10){\line(-1,-1){10}}\put(32,40){\line(-1,-1){30}}\put(22,40){\line(-1,-1){20}}
\put(100,0){\line(1,1){40}}\put(110,0){\line(1,1){30}}\put(120,0){\line(1,1){20}}
\put(100,10){\line(1,1){30}}\put(100,20){\line(1,1){20}}\put(100,30){\line(1,1){10}}
\qbezier(40,38)(108,28)(60,0)
\qbezier(40,32)(100,25.5)(50,0)
\qbezier(40,27)(90,24)(40,3)
\qbezier(100,12)(63,13)(95,40)
\qbezier(100,15)(80,19)(100,28)
\linethickness{0.45mm}
\put(0,20){\line(1,0){40}}\put(40,20){\line(0,1){20}}\put(40,40){\line(1,0){40}}
\put(80,40){\line(0,-1){30}}\put(80,10){\line(1,0){20}}
\put(100,10){\line(0,1){10}}\put(100,20){\line(1,0){40}}
\put(81,30){$w$}
\color{red}
\put(0,20){\line(1,0){22}}
\qbezier(22,20)(31,29)(40,38)
\qbezier(40,38)(80,32)(80,20)
\qbezier(80,20)(80,13)(100,12)
\qbezier(100,12)(104,16)(108,20)
\put(108,20){\line(1,0){29}}
\put(117,21){$u$}
\end{picture}
\caption{Case 2. Plots of $w(x)$ (thick black lines), possible solutions of~\eqref{BEode} (thin black curves), and $u(x)$ (thick red curves) for the case with a transonic rarefaction at $x=0$. The solution $u(x)$ contains no jumps.}
\label{fig:Case2}
\end{center}
\end{figure}
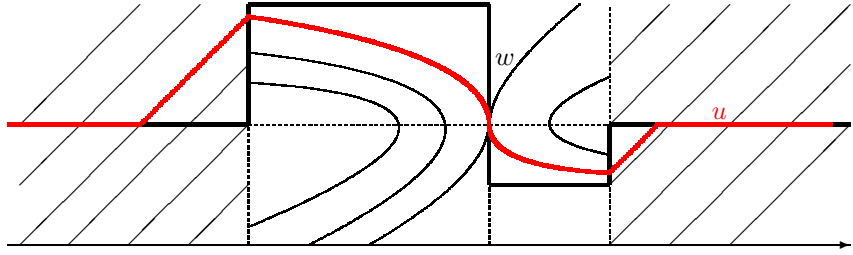

\textbf{Case 3.}
If we use the initial data 
\[
w(x) = \begin{cases} 0.5 & (x<-1) ,\\
0.25 & (-1 <x<0),\\
1 & (0<x<1),\\
0.5 & (x>1),
\end{cases}
\]
the solution of the conservation laws contains a transonic shock initiated at $x=0$. 
The backward Euler solution $u(x)$ contains a jump, see Figure~\ref{fig:Case3}.
Note that there are many places to insert the jump, if no entropy conditions are required.
The unique location of the jump in $u(x)$ is determined by the entropy
conditions in Theorem~\ref{th:characterization} point 1.

\begin{figure}[htbp]
\begin{center}
\setlength{\unitlength}{0.8mm}
\begin{picture}(160,48)(0,0)
\put(0,0){\vector(1,0){160}}
\multiput(0,20)(1,0){160}{\line(1,0){0.5}}
\multiput(40,0)(0,1){40}{\line(0,1){0.5}}
\multiput(80,0)(0,1){40}{\line(0,1){0.5}}
\multiput(120,0)(0,1){40}{\line(0,1){0.5}}
\put(40,38){\line(-1,-1){38}}\put(40,30){\line(-1,-1){30}}\put(40,20){\line(-1,-1){20}}
\put(40,10){\line(-1,-1){10}}\put(32,40){\line(-1,-1){30}}\put(22,40){\line(-1,-1){20}}
\put(120,0){\line(1,1){40}}\put(130,0){\line(1,1){30}}\put(140,0){\line(1,1){20}}
\put(120,10){\line(1,1){30}}\put(120,20){\line(1,1){20}}\put(120,30){\line(1,1){10}}
\qbezier(80,38)(148,28)(100,0)
\qbezier(80,32)(140,25.5)(90,0)
\qbezier(80,27)(130,24)(80,3)
\qbezier(40,20)(40,12)(80,11)\qbezier(40,20)(40,33)(60,40)
\qbezier(80,38)(10,15)(80,13)
\qbezier(80,30)(40,18)(80,16)
\linethickness{0.45mm}
\put(0,20){\line(1,0){40}}\put(40,20){\line(0,-1){10}}\put(40,10){\line(1,0){40}}
\put(80,40){\line(0,-1){30}}\put(80,40){\line(1,0){40}}
\put(120,40){\line(0,-1){20}}\put(120,20){\line(1,0){20}}
\put(51,7){$w$}
\color{red}
\put(0,20){\line(1,0){40}}
\qbezier(40,20)(40,16)(51,13.5)
\put(51,13.5){\line(0,1){13}}
\qbezier(51,26.5)(65,33.5)(80,38)
\qbezier(80,38)(120,32)(120,20)
\put(120,20){\line(1,0){40}}
\put(115,21){$u$}
\end{picture}
\caption{Case 3. Plots of $w(x)$ (thick black lines), possible solutions of~\eqref{BEode} (thin black curves), and $u(x)$ (thick red curves) for the case with  a transonic shock initiated at $x=0$.  The solution $u(x)$ has a jump.  The location of the discontinuity is uniquely determined by the entropy conditions in Theorem~\ref{th:characterization}.}
\label{fig:Case3}
\end{center}
\end{figure}
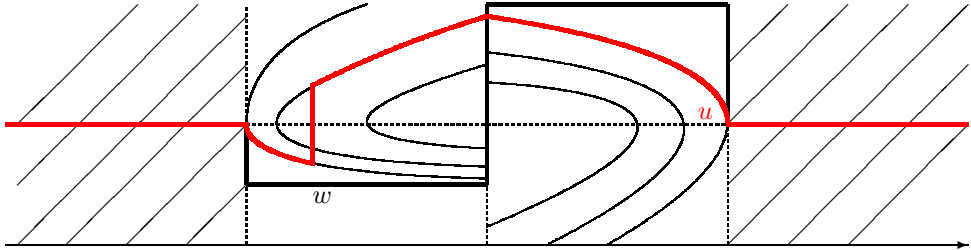

\end{example}

 \begin{example}\label{example2}
 We now give an example of the backward Euler operator for the non viscous 
 conservation law with discontinuous flux. We consider
 \[
 u_t + f(x,u)_x=0, \qquad 
 f(x,u) = \begin{cases}
   f_l(u) = u(1-u) & (x<0), \\
   f_r(u)=2 u (1-u) & (x>0).
 \end{cases}
 \]
We use the following initial data:
\[
w(x) = \begin{cases} 0.5 & (x<-1) ,\\
0.4 & (-1 <x<0),\\
0.7 & (0<x<1),\\
0.5 & (x>1).
\end{cases}
\]
The solution $u=J_\lambda w$ satisfies the ODE
\begin{equation}\label{BEode2}
u'(x) = \frac{w-u}{\lambda f_u(x,u)}.
\end{equation}
In the solution of the conservation law, we have rarefaction waves at $x=\pm1$.
The Riemann problem at $x=0$ is solved with a  stationary jump and a shock 
with positive speed. 
The backward Euler solutions without entropy conditions, are not unique. 
Applying the entropy conditions of Theorem~\ref{th:characterization}, the solution $u(x)$ 
contains two discontinuities, as illustrated in Figure~\ref{fig:Ex2}.
The discontinuity at $x=0$ satisfies the condition in point 4 of 
Theorem~\ref{th:characterization}, while the location of the 
transonic shock satisfies point 1 of Theorem~\ref{th:characterization}.

\begin{figure}[htbp]
\begin{center}
\setlength{\unitlength}{0.8mm}
\begin{picture}(170,48)(-15,0)
\put(-17,18.5){$u=0.5$}
\multiput(0,20)(1,0){160}{\line(1,0){0.5}}
\multiput(40,0)(0,1){40}{\line(0,1){0.5}}
\multiput(80,0)(0,1){40}{\line(0,1){0.5}}
\multiput(120,0)(0,1){40}{\line(0,1){0.5}}
\put(40,38){\line(-1,-1){38}}\put(40,30){\line(-1,-1){30}}\put(40,20){\line(-1,-1){20}}
\put(40,10){\line(-1,-1){10}}\put(32,40){\line(-1,-1){30}}\put(22,40){\line(-1,-1){20}}
\put(120,0){\line(2,1){30}}\put(130,0){\line(2,1){20}}\put(140,0){\line(2,1){10}}
\put(120,10){\line(2,1){30}}\put(120,20){\line(2,1){30}}\put(120,30){\line(2,1){20}}
\qbezier(80,38)(148,28)(100,0)
\qbezier(80,32)(130,25.5)(90,0)
\qbezier(80,27)(120,25)(80,1)
\qbezier(40,20)(40,12)(80,11)\qbezier(40,20)(40,33)(60,40)
\qbezier(80,38)(10,15)(80,13)
\qbezier(80,30)(40,18)(80,16)
\linethickness{0.45mm}
\put(0,20){\line(1,0){40}}\put(40,20){\line(0,-1){10}}\put(40,10){\line(1,0){40}}
\put(80,40){\line(0,-1){30}}\put(80,40){\line(1,0){40}}
\put(120,40){\line(0,-1){20}}\put(120,20){\line(1,0){20}}
\put(51,7){$w$}
\color{red}
\put(0,20){\line(1,0){40}}
\qbezier(40,20)(40,12)(80,11)
\put(80,11){\line(0,-1){10}}
\qbezier(80,1)(82,2)(83,3)
\put(83,2.5){\line(0,1){35}}
\qbezier(83,37.5)(120,32)(120,20)
\put(120,20){\line(1,0){40}}
\put(115,21){$u$}
\end{picture}
\caption{Plots of $w(x)$ (thick black lines), possible solutions of the ODE~\eqref{BEode2} (thin black curves), and $u(x)$ (thick red curves) for Example \ref{example2}. 
The solution $u(x)$ contains two jumps, one at $x=0$, and the other one represents
the transonic shock. 
The location of the discontinuity is uniquely determined by the entropy conditions 
in Theorem~\ref{th:characterization}.}
\label{fig:Ex2}
\end{center}
\end{figure}
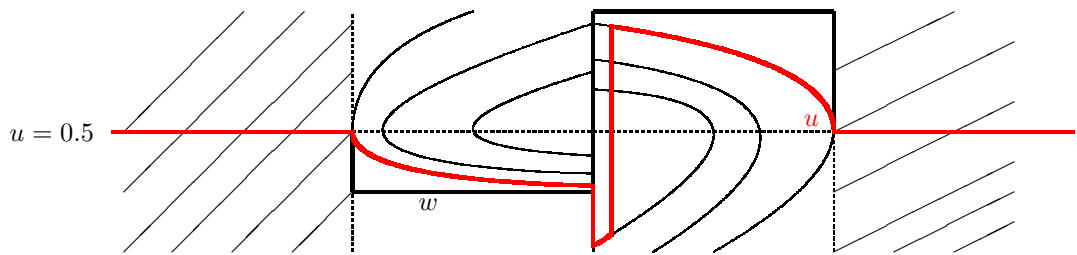
 
 \end{example}

 \bigskip

 We now study the same phenomenon from the point of view of the non
 linear generator of the semigroup.
    At first sight, the evolution equation~\eqref{eq:9} should
    correspond to the operator $B$ defined by
    \begin{equation}
      \label{eq:B_definition}
      \left(u,v\right)\in B\quad\text{ if and only if }\quad
      u,\;v\in\L{1}\left(\mathbb{R},\mathbb{R}\right)
      \text{ and }v = f\left(x,u\right)_{x},
    \end{equation}
    as in the definition of the
    operators $A^{\varepsilon}$. Unfortunately, as we have seen from
    the point of view of the backward Euler operator, the domain of $B$ is
    ``too big'' and it is not an accretive operator,
    therefore the Crandall \& Liggett generation theorem does not apply. 
    We see this in a spatial homogeneous case.

  \begin{example}
    \label{ex:1}
    Consider~\eqref{eq:9} with $f_{l}(u)=f_{r}(u)=f(u)=u(1-u)$, i.e.~a
    classical example for scalar conservation laws:
    \begin{equation}
      \label{eq:classical_ex}
      u_{t}+\left[u\left(1-u\right)\right]_{x}=0.
    \end{equation}
    Let $\phi(x)\in[0,1]$ be a Lipschitz continuous function such that $\phi(x)=0$
    for $x\le-2$ and $\phi(x)=1$ for $x\ge-1$ and define the following
    family of functions parametrized by $\gamma\in[0,1]$
        (See Figure~\ref{fig:fun1}):

\[
      u_{\gamma}(x)=
      \begin{cases}
        \phi(x) &\text{ if }x\le -1,\\
        1&\text{ if }-1<x\le -\gamma,\\
        \frac{1}{2}\left(1-\frac{x}{\gamma}\right)&\text{ if }
        x\in\left]-\gamma,\gamma\right[,\\
        0&\text{ if }x\ge \gamma.
      \end{cases}
\]

  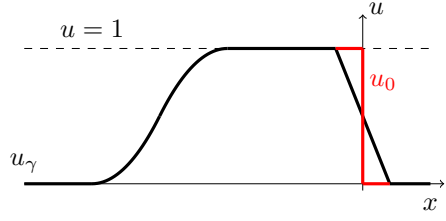
\begin{figure}[htbp]
    \centering
        \begin{tikzpicture}[xscale=0.9,yscale=0.9]
      \draw [line width=0.2,->] (-5,0)-- (1.2,0.);
      \draw [line width=0.2,dashed] (-5,2)-- (1,2.);
      \draw [line width=0.2,->] (0,-0.1)-- (0,2.5);
      \draw[color=black] (1.0,-0.3) node {$x$};
      \draw[color=black] (0.2,2.6) node {$u$};
      \draw[color=black] (-4,2.3) node {$u=1$};
      \draw [black,line width=1.2] (-5,0) -- (-4,0);
      \draw [black,line width=1.2] plot [domain=-4:-3] (\x,{(\x+4)*(\x+4)});
      \draw [black,line width=1.2] plot [domain=-3:-2] (\x,{-(\x+2)*(\x+2) + 2});
      \draw [black,line width=1.2] (-2,2) -- (-0.4,2);
      \draw [black,line width=1.2] (-0.4,2) -- (0.4,0);
      \draw [black,line width=1.2] (0.4,0) -- (1,0);
      \draw[color=black] (-5.0,0.3) node {$u_{\gamma}$};
      \draw [red,line width=1.2] (-0.4,2) -- (0,2);
      \draw [red,line width=1.2] (0,2) -- (0,0);
      \draw [red,line width=1.2] (0,0) -- (0.4,0);
      \draw[color=red] (0.3,1.5) node {$u_{0}$};
    \end{tikzpicture}%
    \caption{Graph of function $u_\gamma(x)$.}
  \label{fig:fun1}
  \end{figure}

    Since $u_{\gamma}$ for $\gamma>0$ is Lipschitz continuous, it
    belongs to the domain of $B$. When $\gamma=0$, $u_{0}$ is
    discontinuous at $x=0$, but $x\mapsto f\left(u_{0}\right)$ is smooth,
    therefore $u_{0}$ also belongs to the domain of $B$.  But we will
    show that
    $u_{0}$ \emph{does not} belong to the domain of $A$ as defined
    in~\eqref{eq:A_domain}.

    For $\gamma\in[0,1]$ we have (see Figure~\ref{fig:fun2})
\[      \left(Bu_{\gamma}\right)(x)=
      \begin{cases}
      \displaystyle  f\left(\phi(x)\right)_{x}&\text{ if }x\le -1,\\
        0&\text{ if }-1<x\le -\gamma,\\
        \displaystyle-\frac{x}{2\gamma^{2}}&\text{ if }-\gamma<x<\gamma,\\
        0&\text{ if }x\ge \gamma,
      \end{cases}
      \qquad 
      \left(Bu_{0}\right)(x)=
      \begin{cases}
        f\left(\phi(x)\right)_{x}&\text{ if }x\le -1,\\
        0&\text{ if }-1<x.
      \end{cases}
      \]

  \begin{figure}[htbp]
    \centering
        \begin{tikzpicture}[xscale=0.8,yscale=0.8]
      \draw [line width=0.2,->] (-5,0)-- (1.2,0.);
      \draw [line width=0.2,->] (0,-0.1)-- (0,2.3);
      \draw[color=black] (1.0,-0.3) node {$x$};
      \draw[color=black] (0.22,2.2) node {$u$};
      \draw [black,line width=1.2] (-5,0) -- (-4,0);
      \draw [black,line width=1.2] plot [domain=-4:-3]
      (\x,{2*4*(\x/2+2)*(1-2*2*(\x/2+2)*(\x/2+2))});
      \draw [black,line width=1.2] plot [domain=-3:-2]
      (\x,{-2*4*(\x/2+1)*(1-2*(1-2*(\x/2+1)*(\x/2+1)))});
      \draw [black,line width=1.2] (-2,0) -- (-0.4,0);
      \draw [black,line width=1.2] plot [domain=-0.3:0.25]
      (\x,{-25 * \x/2 /2});
      \draw [black,line width=1.2] (0.4,0) -- (1,0);
      \draw[color=black] (-5.0,0.3) node {$Bu_{\gamma}$};
      \draw [red,line width=1.2] (-0.4,0) -- (0.4,0);
      \draw[color=red] (0.4,0.3) node {$Bu_{0}$};
    \end{tikzpicture}%
        \caption{Graph of function $(B u_\gamma)(x)$.}
  \label{fig:fun2}
  \end{figure}

    Then for $\gamma>0$, $\lambda\in]0,1]$, we have
    \begin{equation}
      \begin{split}
      \left\|u_{\gamma}-u_{0}\right\|_{\L{1}\left(\mathbb{R},\mathbb{R}\right)}
      &=2\int_{0}^{\gamma}\frac{1}{2}
      \left(1-\frac{x}{\gamma}\right)\;dx=\frac{\gamma}{2},\\
        \left\|u_{\gamma}+\lambda Bu_{\gamma}-\left(u_{0}+\lambda
            Bu_{0}\right)\right\|_{\L{1}\left(\mathbb{R},\mathbb{R}\right)}
        &=2\int_{0}^{\gamma}\left|\frac{1}{2}\left(1-\frac{x}{\gamma}\right)-\lambda\frac{x}{2\gamma^{2}}\right|\;dx,\\
        &=\int_{0}^{\gamma}\left|1-\frac{x}{\gamma}\left(1+\frac{\lambda}{\gamma}\right)\right|\;dx=\frac{\gamma}{2} \cdot
        \frac{1+{\lambda^{2}}/{\gamma^{2}}}{1+{\lambda}/{\gamma}}.
      \end{split}
    \end{equation}
    so that
    \begin{displaymath}
      \left\|u_{\gamma}+\lambda Bu_{\gamma}-\left(u_{0}+\lambda
          Bu_{0}\right)\right\|_{\L{1}\left(\mathbb{R},\mathbb{R}\right)}=
      \frac{1+{\lambda^{2}}/{\gamma^{2}}}{1+{\lambda}/{\gamma}}
      \left\|u_{\gamma}-u_{0}\right\|_{\L{1}\left(\mathbb{R},\mathbb{R}\right)}.
    \end{displaymath}
    Therefore choosing $\lambda\in\left]0,\frac{1}{2}\right]$ and
    $\gamma=2\lambda$ we have
    \begin{displaymath}
      \left\|u_{2\lambda}+\lambda Bu_{2\lambda}-\left(u_{0}+\lambda
          Bu_{0}\right)\right\|_{\L{1}\left(\mathbb{R},\mathbb{R}\right)}=
      \frac{5}{6}
      \left\|u_{\gamma}-u_{0}\right\|_{\L{1}\left(\mathbb{R},\mathbb{R}\right)}.      
    \end{displaymath}
    
    This shows that  $B$ is  not accretive. Furthermore, it does not  satisfies the
    broader condition~\cite[(1.1)]{BP}. Observe that an argument
    similar to the one in Remark~\ref{rem:domains} shows that all
    Lipschitz continuous functions in $D$ are contained in
    $\mathcal{D}\left(A\right)$, therefore $u_{\gamma}\in\mathcal{D}\left(A\right)$
    for any $\gamma\in]0,1]$, hence since $A$ is accretive,  $u_{0}$
    \emph{cannot} belong to the domain of $A$. On the other hand
    some computations show that 
    \begin{displaymath}
      \left\|u_{\gamma}+\lambda Bu_{\gamma}-\left(u_{\bar \gamma}+\lambda
          Bu_{\bar \gamma}\right)\right\|_{\L{1}\left(\mathbb{R},\mathbb{R}\right)}
     \ge 
      \left\|u_{\gamma}-u_{\bar\gamma}\right\|_{\L{1}\left(\mathbb{R},\mathbb{R}\right)}\text{
        for any }\gamma,\bar \gamma\in]0,1],      
    \end{displaymath}
    which is compatible with $A$ being accretive.
  \end{example}

  \begin{remark}
    It is well known that the solution $u(t,x)$ to the Cauchy problem for the
    evolution equation~\eqref{eq:classical_ex} develops discontinuities in finite time, 
even    with smooth integrable
    initial data. 
    If a
    discontinuity travels with a speed different from zero, then
    $\left[u\left(1-u\right)\right]_{x}=-u_{t}$ must contains a Dirac mass,
    hence the solution at time $t$ is not contained in the domain
    $\mathcal{D}\left(A\right)$, see~\eqref{eq:A_values},
    of the generator of the evolution semigroup, but
    only in its closure
    $\overline{\mathcal{D}\left(A\right)}=D$. Therefore, this
    represents a very natural example of a non--linear semigroup for which
    the domain of its generator is not invariant.
  \end{remark}

We note that, in order to apply the
    generation theorem, the domain of $B$ must be
    ``reduced'', and  different ``reductions'' may lead to different generated
    semigroups. 
    The reduction given by~\eqref{eq:A_domain} leads to
    the semigroup of viscous approximations in Theorem~\ref{th:BP_limit}. 
    This reduction can also lead to Kru\u zkov entropy inequalities, see~\cite{C72}
    for the multidimensional case with smooth fluxes, 
    or~\cite{andreianov1} for~\eqref{eq:9}. 
    Kru\u zkov 
    entropy inequalities can also be used to define different reductions
    which gives correspondingly different semigroups in~\cite{andreianov1},
    referred to as ``germs".

    What happens  if the dependence of the flux $f$ on the spatial
    variable $x$ 
    is more irregular? 
 In~\cite{BGS}, using Theorem~\ref{th:BP_limit}
    as a building block, 
    existence and uniqueness of the vanishing viscosity limit for fluxes
    $f\left(t,x,\omega\right)$ with general
    $\bv$ regularity with respect to the variables $(t,x)$
    is obtained. The result in~\cite{BGS} is based on comparison estimates for
    solutions to the corresponding Hamilton--Jacobi equations.

The $\bv$ regularity on the flux is an essential assumption, as shown in the following
counter example. 
    Suppose that the map $x\mapsto f(x,\omega)$ is
    $\L\infty$ 
    but with unbounded variation.
    In this case the domain of the operator
    $A u = f(x,u)_x$ may not be dense in $\L1$.
    For every $\varepsilon>0$, the viscous approximations
    \begin{equation}
      \label{41}
      u^\varepsilon_t + f(x,
      u^\varepsilon)_x~=~\varepsilon\,u^\varepsilon_{xx}\,,\qquad 
      u^\varepsilon (0,x)=\bar  u(x),
    \end{equation}
     are still well
    defined for any initial data $\bar
    u\in\L1\left(\mathbb{R},\mathbb{R}\right)$, according to
    Theorem~\ref{th:gen_linfty_f}.
    However, they may not converge to a (weakly) continuous function
    of time $t\mapsto u(t)$.  In the next example, 
    we show that one could have
    \begin{equation}
      \label{42}
    \lim_{t\to 0+} \left( \lim_{\varepsilon\to 0}
      u^\varepsilon(t,\cdot)\right)~\not=~\bar u\,.
  \end{equation}

  \begin{example}
    \label{ex:second}
    Let $\mathbb{Q}=\left\{q_{i}\right\}_{i=1}^{+\infty}$ be an
    enumeration of the rational numbers, $\kappa >0$, and $V$ be the
    open set defined by
    \[V=\displaystyle{\bigcup_{i=1}^{+\infty}\left]q_{i}-\kappa 2^{-\left(i+1\right)},
      q_{i}+\kappa 2^{-\left(i+1\right)}\right[},\]
       so that $\m\left(V\right)\le  \kappa.$
   Define the closed set $K=\mathbb{R}\setminus V$ and the
  function $\alpha=\chi_{K}$ i.e. the characteristic
  function of the set $K$. Observe that $K$ is totally disconnected,
  that any rational number
  $q\in\mathbb{Q}$ has a neighborhood in which $\alpha$ is identically
  zero and that $\alpha$ has unbounded total variation on any
  interval with length greater than $\kappa$.

  \begin{theorem}
    Any weak solution $u\in\L\infty\left(\left[0,T\right]\times
      \mathbb{R},\mathbb{R}\right)$ to the conservation law
    \begin{equation}
      \label{eq:lastCL}
      u_{t}+f(x,u)_{x}=0,\qquad
      f(x,u)~=~\alpha(x)\, u(1-u)
    \end{equation}
    such that the map $t\mapsto
    u(t,\cdot)$ is continuous from $\left[0,T\right]$ into
    $\L\infty\left(\mathbb{R},\mathbb{R}\right)$ endowed with the
    weak$^{*}$ topology is constant in time, $u(t,x)=\bar u(x)$, and
    must satisfy $\bar u(x)\in\left\{0,1\right\}$ almost everywhere in $K$.
  \end{theorem}
  \begin{proof}
    For $\varepsilon>0$, define the convolution kernels $\lambda_{\varepsilon}$ as
    \begin{equation}
      \label{eq:convKer}
        \lambda_{\varepsilon}(x)=\frac{1}{\varepsilon}
        \lambda\left(\frac{x}{\varepsilon}\right), \quad
        \lambda\in\C\infty\left(\mathbb{R},[0,1]\right),\quad
      \int_{\mathbb{R}}\lambda(x)\; dx =1,\quad 
      \lambda(x)=0\;
        \forall x\not\in \left[-1,1\right].
    \end{equation}
    We let
    \[
    a_{\varepsilon}(x)=\int_{-\infty}^{x}\lambda_{\varepsilon}
        \left(\xi\right)\; d\xi.
    \]
    Fix two rational numbers $r<q$, a time $\tau>0$, small positive
    $\varepsilon,\gamma>0$ and evaluate~\eqref{eq:lastCL} using the
    test function
    \begin{displaymath}
      \varphi\left(t,x\right)=\left(a_{\varepsilon}\left(x-r\right)
        -a_{\varepsilon}\left(x-q\right)\right)\left(a_{\gamma}\left(t-\gamma\right)
        -a_{\gamma}\left(t-\tau-\gamma\right)\right),
    \end{displaymath}
    we get
    \begin{equation}
      \begin{split}
        &\int_{0}^{\tau+2\gamma}\int_{r-\varepsilon}^{q+\varepsilon}
        u(t,x)\left(a_{\varepsilon}\left(x-r\right)
          -a_{\varepsilon}\left(x-q\right)\right)\left(\lambda_{\gamma}\left(t-\gamma\right)
          -\lambda_{\gamma}\left(t-\tau-\gamma\right)\right)\;dtdx\\
        +&\int_{0}^{\tau+2\gamma}\int_{r-\varepsilon}^{q+\varepsilon}
        f\left(x,u(t,x)\right)\left(\lambda_{\varepsilon}\left(x-r\right)
          -\lambda_{\varepsilon}
          \left(x-q\right)\right)\left(a_{\gamma}\left(t-\gamma\right)
          -a_{\gamma}\left(t-\tau-\gamma\right)\right)\;dtdx=0.
      \end{split}
    \end{equation}
    If $\varepsilon$ is sufficiently small the supports of
    $\lambda_{\varepsilon}\left(x-r\right)$ and of
    $\lambda_{\varepsilon} \left(x-q\right)$ are contained in $V$
    where $f(x,u(t,x))$ vanishes. Therefore, the previous equality
    becomes
    \begin{displaymath}
      \int_{0}^{\tau+2\gamma}\int_{r-\varepsilon}^{q+\varepsilon}
      u(t,x)\left(a_{\varepsilon}\left(x-r\right)
        -a_{\varepsilon}\left(x-q\right)\right)\left(\lambda_{\gamma}\left(t-\gamma\right)
        -\lambda_{\gamma}\left(t-\tau-\gamma\right)\right)\;dtdx=0.
    \end{displaymath}
    Letting $\varepsilon$ tend to zero we obtain
    \begin{displaymath}
      \int_{0}^{2\gamma}\lambda_{\gamma}\left(t-\gamma\right)
      \int_{r}^{q}
      u(t,x)\;dxdt=
      \int_{\tau}^{\tau+2\gamma}\lambda_{\gamma}\left(t-\tau-\gamma\right)
      \int_{r}^{q}
      u(t,x)\;dxdt.
    \end{displaymath}
    The weak$^{*}$ continuity assumption implies that the map
    $t\mapsto \int_{r}^{q} u(t,x)\;dx$ is continuous, therefore we can
    take the limit as $\gamma\to 0$ and get
    \begin{displaymath}
      \int_{r}^{q}
      u(\tau,x)\;dx
      =
      \int_{r}^{q}
      u(0,x)\;dx,\quad\text{ for any }r<q,\;\text{ with
      }r,q\in\mathbb{Q} \text{ and }\tau\in\left[0,T\right].          
    \end{displaymath}
    This implies that $u(t,x)=u(0,x)\,\dot=\,\bar u(x)$ must be constant
    in time as a function from $[0,T]$ into
    $\L\infty\left(\mathbb{R}, \mathbb{R}\right)$.
    From~\eqref{eq:lastCL}, we have
    \begin{displaymath}
      \left[\alpha(x)\bar u(x)\left(1-\bar u(x)\right)\right]_{x}=0,
      \quad
      \Longrightarrow\quad \alpha(x)\bar u(x)\left(1-\bar u(x)\right)=C
    \end{displaymath}
    for some constant $C\in\mathbb{R}$. But $\alpha$ vanishes on the
    set $V$, therefore $C=0$. Finally, since $\alpha(x)=1$ for any
    $x\in K$, then $\bar u(x)\in\left\{0,1\right\}$ a.e. $x\in K$.    
\end{proof}

 \begin{remark}
    As a consequence of this theorem, only initial data $\bar u$ that
    satisfy $\bar u(x)\in \left\{0,1\right\}$ almost everywhere on $K$
    have a solution to the Cauchy problem
    \begin{displaymath}
      \begin{cases}
        u_{t}+f(x,u)_{x}=0\\
        u(0,x)=\bar u(x)
      \end{cases}
    \end{displaymath}
    which depends continuously on time.
    Hence,
    if $b-a>\kappa$, and $\bar u =
    \frac{1}{2}\chi_{\left[a,b\right]}$, 
    the previous Cauchy problem
    cannot have a weak solution $u\in\L\infty\left(\left[0,T\right]\times
      \mathbb{R},\mathbb{R}\right)$ such that the map $t\mapsto
    u(t,\cdot)$ is continuous from $\left[0,T\right]$ into
    $\L\infty\left(\mathbb{R},\mathbb{R}\right)$ endowed with the
    weak$^{*}$ topology.
 \end{remark}

  We remark that the initial condition $\bar
  u=\frac{1}{2}\chi_{\left[a,b\right]}$ 
  does not
  lie in the closure of the domain of the operator
  $Au\doteq f(x,u)_x$.  This can be checked by showing that, if
  $\|u-\bar u\|_{\L1\left(\mathbb{R},\mathbb{R}\right)}<\rho$
  with $0<\rho<\frac{b-a-\kappa}{4}$, then the function
  $x\mapsto f(x, u(x)) $ has unbounded variation.
  Indeed, if $\|u-\bar
  u\|_{\L1\left(\mathbb{R},\mathbb{R}\right)}<\rho$, then,
  setting $B=\left\{x\in\left[a,b\right]: 
        \frac{1}{4}\le u(x)\le \frac{3}{4}\right\}$
\begin{displaymath}
    \m\left(B\right)=b-a-
    \m\left(\left\{x\in\left[a,b\right]: \left|u(x)-\frac{1}{2}\right|>
        \frac{1}{4}\right\}\right)
    \ge b-a -4\|u-\bar u\|_{\L1_{\L1\left(\mathbb{R},\mathbb{R}\right)}}>\kappa.
  \end{displaymath}
  So that $\m\left(K\cap B\right)>0$.
We now have
\begin{displaymath}
  f(x, u(x))
  \begin{cases}
    \ge \frac{3}{16}&\text{ if }x\in B\cap K\\
    =0 &\text{ if }x\in V.
  \end{cases}
\end{displaymath}
Since for any two points $x_1, x_2\in B\cap K$
we can find an interval contained in $V$ between them,  the total
variation of $f(x,u)$ is infinite and $u\notin \D(A)$.  It is
thus clear that the classical theory of contractive semigroups
\cite{CL} cannot be applied here.  
\end{example}

%
%
%
%
%
%
%


\bigskip

\noindent\textbf{Acknowledgment:} The present work was supported by
the PRIN~2015 project \emph{Hyperbolic Systems of Conservation Laws
  and Fluid Dynamics: Analysis and Applications} and by GNAMPA 2017
project \emph{Conservation Laws: from Theory to Technology}.
The authors  would like to thank the first anonymous referee for suggesting
a discussion on the comparison between vanishing viscosity 
and the adapted entropies,  
which led to the creation of Section 6, 
and the second anonymous referee for carefully reading the manuscript
and providing many useful suggestions.

\end{document}